\documentclass[smallextended]{svjour3}    

\smartqed 
\usepackage{amssymb}
\usepackage{amsxtra}
\usepackage{graphicx}
\usepackage{caption}
\usepackage{afterpage}
\usepackage{amssymb}
\usepackage{amsfonts}
\usepackage{xcolor}
\usepackage{lineno}

\usepackage{amsthm}
\usepackage[T1]{fontenc}    
\usepackage[utf8]{inputenc} 
\usepackage[algo2e,ruled,linesnumbered,vlined]{algorithm2e} 
\usepackage{soul}

\newcommand\mtiny[1]{\mbox{\tiny\ensuremath{#1}}}
\usepackage{lipsum}
\usepackage{ulem}

\makeatletter
\let\cl@chapter\undefined
\makeatletter

\usepackage{cleveref}
\crefname{Lemma}{Lemma}{Lemmas}

\crefname{Lemma}{Lemma}{Lemmas}

\usepackage{mathtools}

\usepackage{microtype}
\usepackage{comment}
\usepackage{makecell}

\newcommand{\cR}{\mathbb{R}}

\newtheorem{obs}[theorem]{Observation}

\newcommand{\eps}{\epsilon}

\newcommand{\be}{\begin{equation}}
\newcommand{\benn}{\begin{equation*}}
\newcommand{\ee}{\end{equation}}
\newcommand{\eenn}{\end{equation*}}
\newcommand{\bes}{\begin{split}}
\newcommand{\ens}{\end{split}}

\newcommand{\tf}{{\tilde f}}
\newcommand{\tA}{{\tilde A}}
\newcommand{\tg}{{\tilde g}}
\newcommand{\tc}{{\tilde c}}
\newcommand{\tW}{{\tilde W}}
\newcommand{\tZ}{{\tilde Z}}
\newcommand{\tphi}{{\tilde \phi}}
\newcommand{\mE}{{\mathcal E}}

\newcommand{\df}{{\delta_f}}
\newcommand{\dA}{{\delta_A}}
\newcommand{\dg}{{\delta_g}}
\newcommand{\dc}{{\delta_c}}

\newcommand{\ef}{{\epsilon_f}}
\newcommand{\eA}{{\epsilon_A}}
\newcommand{\eg}{{\epsilon_g}}
\newcommand{\ec}{{\epsilon_c}}

\newcommand{\lam}{\lambda}

\newcommand{\grad}{\nabla}

\newcommand{\lh}{{\hat l}}
\newcommand{\pred}{{\mathbf{pred}}}
\newcommand{\ared}{{\mathbf{ared}}}
\newcommand{\vpred}{{\mathbf{vpred}}}
\newcommand{\hpred}{{\mathbf{hpred}}}

\newcommand{\EAcI}{{\mathcal{E}_{Ac}^{I}}}

\usepackage{lmodern}
\usepackage[T1]{fontenc}
\usepackage{pdfpages}

\newcommand{\jn}[1]{{\color{blue}#1}}
\newcommand{\shig}[1]{{\color{magenta}#1}}

\begin{document}

\title{A Trust-Region Algorithm for Noisy Equality Constrained Optimization\thanks{Sun was supported by NSF grant DMS-1620022. Nocedal was supported by AFOSR grant FA95502110084 and ONR grant N00014-21-1-2675.}
}

\titlerunning{A Trust-Region Algorithm for Noisy Equality Constrained Optimization}    

\author{Shigeng Sun\and Jorge Nocedal 
}


\institute{Shigeng Sun \at
       Department of Engineering Sciences and Applied Mathematics, Northwestern University, Evanston, IL, USA \\
       \email{shigengsun2024@u.northwestern.edu}      
      \and
      Corresponding author: Jorge Nocedal \at
       Department of Industrial Engineering and Management Sciences, Northwestern University, Evanston, IL, USA\\
       \email{j-nocedal@northwestern.edu}
}

\date{Received: 02 September 2024 / Accepted: }

\maketitle

\begin{abstract}

This paper introduces a modified Byrd-Omojokun (BO) trust region algorithm to address the challenges posed by noisy function and gradient evaluations. The original BO method  was designed to solve equality constrained problems and it forms the backbone of some interior point methods for general large-scale constrained optimization, such as {\sc {\sc knitro}} \cite{ByrdNoceWalt06}. A key strength of the BO method is its robustness in handling problems with rank-deficient constraint Jacobians. The  algorithm proposed in this paper introduces a new criterion for accepting a step  and for updating the trust region that makes use of an estimate in the noise in the problem.  The analysis presented here gives conditions under which the iterates converge to regions of stationary points of the problem,  determined by the level of noise. This analysis is more complex than for line search methods because the trust region carries (noisy) information from previous iterates. Numerical tests illustrate the practical performance of the algorithm.

\end{abstract}
\vspace{1cm}




\newpage
\section{ Problem Statement }
\label{sec:algorithmBO}

Our goal is to propose a variant of the Byrd-Omojokun algorithm \cite{ByrdGilbNoce00}  designed to handle problems where noise affects the function and constraint evaluations. The Byrd-Omojokun (BO) algorithm is a sequential quadratic programming (SQP) method for solving equality constrained optimization problems. It employs trust regions to safeguard the iteration and uses a non-smooth merit function to guide the iterates to stationary points of the problem. The algorithm  is robust even when the Jacobian of the constraints is rank deficient, and can efficiently solve very large problems. 
The BO algorithm has been incorporated or adapted into various methods  for nonlinearly constrained optimization \cite{HeinkRidz}, and is integral to the {\sc {\sc knitro}} software package \cite{ByrdNoceWalt06} . 
 
The problem under consideration is:
\begin{align} \label{problem} 
& \min_x  f(x)    \\
& \text{ s.t.  }  \,  c(x)   = 0 ,\nonumber
\end{align}
where $f:{\cR}^n \rightarrow \cR$ and $c:\cR^n \rightarrow \cR^m$ are smooth functions with gradient and Jacobian denoted, respectively, as
\be  g_k = \grad f(x_k)\in \cR^{n*1}, \quad A_k = \grad c(x_k)\in \cR^{m*n}.\ee

This paper concerns the case where the above quantities cannot be evaluated exactly but we have access to noisy observations denoted as 
\be
\tilde f(x)=  f(x) + \delta_f(x),\quad
\tilde c(x)  = c(x) + \delta_c(x) ;
\ee
\be \tilde g_k = \grad f(x_k) + \delta_g(x),\quad \tilde A_k = \grad c(x_k) + \delta_A(x);\ee
where $\delta_f(x), \delta_c(x),\delta_g(x),\delta_A(x)$ denote noise or computational errors. We  define the Lagrangian as
\be \tilde L(x,\lam)  = \tilde f(x) - \lam^T\tilde  c(x).\ee

Much recent research has focused on developing optimization algorithms for noisy constrained problems of the form \eqref{problem}. While there has been significant interest in this area, trust region methods have received comparatively less attention.
The fact that the trust region includes information from previous iterations makes the analysis in the noisy setting 
 more challenging than for line search methods. Our results are of significant generality in that they also cover the case when the Jacobian of the constraints loses rank.
 This paper builds upon the framework developed in \cite{sun2023trust} for studying trust region methods for unconstrained optimization.

\medskip\noindent\textit{Notation.} Throughout the paper, $\| \cdot\|$ denotes the $\ell_2$-norm. As usual, we abbreviate $f(x_k)$ as $f_k$, etc.

\subsection{Literature Review}
Nonlinear optimization problems with equality constraints arise in a wide range of disciplines, and a variety of line search and trust region methods have been designed to solve them. Among trust region methods, notable approaches include those proposed by Celis-Dennis-Tapia \cite{celis1984trust}, Yuan-Powell \cite{powell1990trust},  Vardi \cite{vardi1985trust}. However, the Byrd-Omojokun algorithm \cite{Omoj89}  stands out, as it strikes the right balance between robustness and scalability  \cite{LaleNocePlan98}. This method plays an important role in modern software for general nonlinearly constrained optimization, as mentioned above.

Recently, there has been increasing interest in adapting trust region methods to solve \textit{unconstrained} problems  with noise in the objective functions and derivatives \cite{sun2023trust,cao2023first,toint2021TRqEDAN,krejic2024non,curtis2022fully,chen2018stochastic}. Adaptations are necessary since classical trust region methods for deterministic optimization can struggle or even fail in this setting \cite{sun2023trust}. 
For example, \cite{sun2023trust} modifies the trust region ratio test by relaxing its numerator and denominator based on noise level (assumed to be bounded), and establishes convergence guarantees. Similar approaches are found in \cite{boumal2023introduction} and \cite{hansknecht2024convergence}, with a heuristic in \cite{conn2000trust}. Additionally, \cite{cao2023first} proposes modifying only the numerator in the trust region ratio test along with other imposed algorithmic conditions, and establishes convergence rates results in high probability.
These methods typically do not require diminishing noise, but the technique proposed in  \cite{toint2021ARqpEDA2} can take advantage of that possibility.

There are few studies on methods for noisy \textit{constrained} optimization   \cite{oztoprak2023constrained,curtis2024stochastic,berahas2023sequential,fang2024fully}. In \cite{oztoprak2023constrained}, a line search SQP algorithm relaxes the descent condition to accommodate noise, ensuring convergence. \cite{berahas2023sequential} presents a step-search SQP algorithm employing a technique different from line searches and trust regions, while \cite{fang2024fully} introduces an approach that has some similarities with the Byrd-Omojkun method, and establish convergence in the stochastic setting.

Most research assumes full-rank Jacobians \cite{fang2024fully,oztoprak2023constrained,berahas2023sequential}, except \cite{berahas2023stochastic}, which also considers non-biased gradient estimates. One of the hallmarks of trust region methods is their ability to deal with rank-deficient Jacobians, see e.g. \cite{conn2000trust,byrd2000trust}, for a discussion of the deterministic setting. {Our work distinguishes itself from previous studies by considering a standard trust region method for equality-constrained optimization, as opposed to modifications that eliminate history by either using a predetermined trust region schedule or defining the trust radius as a multiple of the current gradient norm.}

\section{The Algorithm}
At a current iterate $x_{k}$,  the algorithm utilizes a trust region radius $\Delta_{k}$, Lagrange multipliers $\lambda_{k}$, and an approximation $\tilde W\left(x_{k}, \lambda_{k}\right)$ to the Hessian of the Lagrangian $\tilde L(x_k,\lambda_k)$. With this information, the aim is to generate a  step $p_{k}$ by solving the subproblem
\begin{align}  \label{ideal}
\min_p\quad  &\tg_{k}^{T} p+\frac{1}{2} p^{T} \tilde W\left(x_{k}, \lambda_{k}\right) p\\
 \text{ subject to } \quad &\tilde A_k p + \tilde c_k = 0  \label{linearc} \\
\quad &\|p \|\leq \Delta_k . \label{trustc}
\end{align}
However, this problem may be infeasible: by restricting the size of the step, the trust region may preclude satisfaction of the linear constraints \eqref{linearc}. To address this difficulty, the Byrd-Omojokun method performs the step computation in two stages. First, a normal step determines a desirable level of feasibility which is then imposed upon subproblem \eqref{ideal}-\eqref{trustc}. We now discuss the adaptation of this method to the noisy setting.

\medskip
\textit{Normal Step}: The goal of this step is to find an acceptable level of feasibility in the linear constraints \eqref{linearc}.  To this end, we choose a contraction parameter $\zeta\in (0,1)$ and compute $v_k$ which solves:

\begin{align}
\min_v \quad  & \| \tA_k v  + \tc_k\| \label{vprob} \\
 \text{ subject to } \quad & \|  v \| < \zeta \Delta_k . \label{vprobtr}
\end{align}

\medskip
\textit{Full Step}. With $v_k$ at hand, we can now define the relaxed version of the subproblem (6) as follows

\begin{align}
\label{hprob_exact}
\min_p\quad  &\tg_{k}^{T} p+\frac{1}{2} p^{T}\tW\left(x_{k}, \lambda_{k}\right) p\\
 \text{ subject to } \quad &\tA_k p + \tc_k = \tA_k v_k + \tc_k \nonumber \\
\quad &\|p \|\leq \Delta_k . \nonumber
\end{align}
This problem is always feasible and we denote a 
solution by  $p_k$. 
 In this paper we assume that these two subproblems are solved exactly, but to establish the convergence results presented below, it  suffices to compute approximate solutions that yield a fraction of Cauchy decrease; see e.g. \cite{mybook}.

The BO method is a primal method that uses least squares multiplier estimates. They are defined as a solution to the problem
\be {\label{lm} \min_\lambda \| \tilde g_k - \tilde A_k^T \lam \|^2}. \ee

\textit{Step Acceptance and Trust Region Update.}
To determine if the step is acceptable, the BO algorithm  uses the nonsmooth merit function
\be \label{meritf} \tphi(x,\nu) = \tilde f(x) + \nu \|\tilde c(x) \| ,\ee
where $\nu$ is called the penalty parameter.  We construct a model of $\tphi(\cdot,\nu_k)$ at $x_k$ as
\be
\label{model} 
m_k(p) =\tilde f(x_k) + p^{T} \tilde g_{k}+\tfrac{1}{2} p^{T} \tilde W_{k} p+\nu_k\left\|\tilde A_{k} p+\tilde c_{k}\right\| .
\ee
We define the predicted reduction in the merit function $\tilde \phi(\cdot, \nu)$ to be the change in the model $m_k$ produced by a step $p_k$:
\begin{equation}
    \pred_k(p_k)  = m_k(0)- m_k(p_k). \label{pred} 
\end{equation}
 Before testing step acceptance, we update the penalty parameter $\nu_k$  to ensure that $\pred_k(p_k)$ is sufficienlty positive. Given a scalar { ${\pi_1} \in (0,1)$,} the new penalty parameter $v_k$ is chosen large enough so that (see  \cite[eq(2.35)]{ByrdGilbNoce00})
 \be \pred_k(p_k) > {\pi_1} \nu_k\vpred_k(p_k),\label{merit_rule}\ee
 where 
\begin{equation} \vpred_k(p_k)  = \|\tc_k\| - \|\tA_k p_k+\tc_k\| \label{vpred}
\end{equation}
is  the reduction in the objective of the normal problem. {It is easy to see from the definitions \eqref{pred} and \eqref{vpred} that there always exist large enough $\nu_k$ that satisfy \eqref{merit_rule}.}

Having chosen the penalty parameter $\nu_k$, we test whether the step $p_k$ is acceptable. As in any trust region algorithm, {this test is based on} the ratio between the actual and predicted reduction in the merit function, where the former is defined as 
\begin{equation}  
\ared_k(p_k)  = \tphi(x_k,\nu_k)- \tphi(x_k+p_k,\nu_k). \label{ared}
\end{equation}
Due to the presence of noise, we introduce some slack in this test. We define a relaxed ratio as
\be \label{rhodef} \rho_{k} = \frac{\ared_k+\xi(\eps_f + \nu_k\eps_c) }{\pred_k +\xi(\eps_f + \nu_k\eps_c)},
\ee
where $\xi$ is a constant specified below, {and $\eps_f$ and $ \eps_c$ denote the noise level in the function and constraints, as defined in \eqref{noisebBO}.}
We use the value of $\rho_k$ to determine whether a step is acceptable and whether the trust region radius should be adjusted.

\subsection{Specification of the Algorithm} We are now ready to state the variant of the Byrd-Omojukun algorithm designed to solve the noisy equality constrained optimization problem \eqref{problem}. The only requirement we impose on the Hessian approximation $\tilde W_k$ is that it be a bounded symmetric matrix.

\begin{algorithm2e}[H] \label{algorithmBO} 
\SetAlgoLined
 Initialize {$x_0,\nu_{-1}$}, $\Delta_{0}$. $k=0$; \\
 {Input: $\eps_f, \eps_c$ (noise level)} \\
 Choose constants {${\pi_1}, \pi_0, \zeta$, all in (0,1),} and
 $\tau>1$;\\
 Set relaxation parameter: $\xi = \frac{2}{1-\pi_0}$;\\
 

 \While{a termination condition is not met}{
 Evaluate $\tf_k, \tc_k, \tg_k, \tA_k$;\\ 
 Solve \eqref{lm} for $\lam_k$, compute $ \tW_k $;\\
 Solve  subproblem \eqref{vprob} for $v_{k}$ and \eqref{hprob_exact} for $p_{k}$;\\
 Evaluate $\pred_k$ and $\vpred_k$ by \eqref{pred}, \eqref{vpred};\\
 Set: $\nu_k = \nu_{k-1}$;\\
 \While{$\pred_k \leq {\pi_1} \nu_k \vpred_k$}{
    $\nu_k = \tau \nu_k$;\\
    Re-evaluate $\pred_k$;
}
 Evaluate $\ared_k$ by \eqref{ared} ;\\
 Compute {\be\rho_{k} = \frac{\ared_k+\xi(\eps_f + \nu_k\eps_c) }{\pred_k +\xi(\eps_f + \nu_k\eps_c)}; \label{relaxed}\ee} \\

 \uIf{$\rho_{k}>\pi_0$}{
  {$x_{k+1} = x_k+p_k $}, $\Delta_{k+1} = \tau \Delta_{k}$;
 }\Else{
  {$x_{k+1} = x_k$}, $\Delta_{k+1} =  \Delta_{k}/\tau$;
 }
 Set $k \leftarrow k+1$;
 }
\caption{The Noise Tolerant Byrd-Omojukun Algorithm}
\end{algorithm2e}
We note that line 7 requires the solution of two trust region problems. In practice, this can be done inexactly, as mentioned above, allowing the BO method to scale into the tens of thousands of variables \cite{ByrdNoceWalt06}. The analysis presented here is applicable to both the exact and inexact cases. 

In the next section, we establish global convergence properties of Algorithm~1 to a region of stationary points of the problem. In section~\ref{numerical}, we present numerical experiments illustrating the behavior of the algorithm.

\section{Global Convergence}
We make the following assumptions about the problem, the  noise (or errors), and the iterates.

\medskip\noindent
\textbf{Assumption 1:} $f(x)$, $c(x)$ are $L_f$ and $L_c$--smoothly differentiable, respectively.\\ \\
\textbf{Assumption 2:} The sequences $\{\tA_k\}$, $\{\tW_k\}$, $\{\tc_k\}$ generated by the algorithm are bounded: \textit{i.e.} $\forall k$: 
 \be \|\tA_k\| \leq M_A;\quad \|\tW_k\|\leq M_W;\quad \| \tc_k\|\leq M_c,\label{constants}\ee
  for some constants $M_A$, $M_W$, $M_c$. 
 Furthermore, the sequence $\{\tilde f_k\}$ is bounded below.\\
 
\noindent\textbf{Assumption 3:} There exist constants $\ef$, $\ec$, $\eg$ and $\eA$ such that, for all $x \in \mathbb{R}^n$,
\begin{equation}  \label{noisebBO}
| \df(x)| \leq \ef, \quad \| \dc(x)\| \leq \ec ,\quad \| \dg(x) \| \leq \eg ,\quad \| \dA(x) \| \leq \eA .
\end{equation}
In other words, we assume that noise (or errors) are bounded, which is the case in many practical applications; see e.g. the discussion in \cite{lou2024noise}. {We refer to $\eps_f, \eps_c$ as the \textit{noise level} in the problem.}
.

\subsection{{Reduction in the Feasibility Measure}}
In this section, we show that Algorithm~1 is able to reduce {a stationarity measure of feasibility} to a level consistent with the noise level in the functions.
The first result follows from classical trust region convergence theory; see e.g. \cite{conn2000trust,mybook}.

\begin{lemma}
The step $p_k$ computed by Algorithm~\ref{algorithmBO} satisfies
\be \vpred_k(p_k) = \|\tc_k\| - \|\tA_k p_k+\tc_k\|\geq \frac{\|\tA_k^T \tc_k\|}{2\|\tc_k\|}\min\left(\zeta\Delta_k,\frac{\|\tA_k^T \tc_k\|}{\|\tA_k^T\tA_k\|}\right).\label{lem1_res}\ee
\end{lemma}
%
\noindent The next lemma shows that $m_k$ is an accurate model of the merit function when $\Delta_k$ is small. 

\begin{lemma}[Accuracy of the Model of the Merit Function]
Under Assumptions~1-3, \be \left| \ared_k(p_k) -\pred_k(p_k)\right|\leq M_L(\nu_k) \Delta_k^2 + (\eg + \nu_k \eA)\Delta_k + 2(\ef + \nu_k \ec),\label{lem2_res}\ee
where
\be M_L(\nu_k) =\max(L_f + M_W,\ \nu_k L_c). \label{def_ML}\ee
\end{lemma}
\begin{proof}
From \cref{model,pred,vpred} we have: 
\be \label{pred_vs_vpred}
\pred_k(p_k) =  -p_k^{T} \tg_{k}-\frac{1}{2} p_k^{T} \tW_{k} p_k + \nu_k  \vpred_k(p_k).\ee
Using this fact, and recalling Assumptions 1-3, we have
\be
\begin{split} \nonumber
&\left| \ared_k(p_k) -\pred_k(p_k)\right|\\
 =&\left| [\tphi(x_k) - \tphi(x_{k+1}) ] -[m_k(0)-m_k(p_k)]\right|\\
 = & \left| \tf_k-\tf_{k+1} + \nu_k [\|\tc_k\|-\|\tc_{k+1}\|] - \left[-p_k^{T} \tg_{k}-\frac{1}{2} p_k^{T} \tW_{k} p_k + \nu_k  \vpred_k(p_k)\right ]   \right|\\
 \leq & \left| f_k-f_{k+1} + \nu_k [\|c_k\|-\|c_{k+1}\|] - \left[-p_k^{T} g_{k}-\frac{1}{2} p_k^{T} \tW_{k} p_k + \nu_k  \vpred_k(p_k)\right ]   \right| + ...\\
 & ... + \left| \df(x_k) + \df(x_{k+1}) + p_k^T \dg(x_k) +  \nu_k[\|\dc(x_k)\| + \|\dc(x_{k+1})\|] \right| \\
  \leq & \left| \int_{0}^{1}\left[g\left(x_{k}+t p_{k}\right)-g_k\right]^{T} p_{k} d t + \nu_k [ \|A_k^Tp_k+c_k\|-\|c_{k+1}\|] +\frac{1}{2} p_k^{T} \tW_{k} p_k  \right| + ...\\
  &... + 2(\ef + \nu_k \ec) + \eg\|p_k\| + \nu_k \|\dA(x_k)^T p_k\|  \\
  \leq & \frac12{(L_f+M_W+\nu_k L_c)} \|p_k\|^2  + \eg\|p_k\| + \nu_k \|\dA(x_k)^T p_k\| +  2(\ef + \nu_k \ec) \\
  \leq& \frac12{(L_f+M_W+\nu_k L_c)} \Delta_k^2 + (\eg + \nu_k \eA)\Delta_k + 2(\ef + \nu_k \ec)\\
  \leq& \max(L_f + M_W,\ \nu_k L_c) \Delta_k^2 + (\eg + \nu_k \eA)\Delta_k + 2(\ef + \nu_k \ec)\\
  =& M_L(\nu_k) \Delta_k^2 + (\eg + \nu_k \eA)\Delta_k + 2(\ef + \nu_k \ec).
 \end{split}\ee
\end{proof}


For economy of notations we define, for any given iterate $k$,
\be \label{defi1} \mE_v(k):= \frac{\xi M_c}{{\pi_1} \zeta}  (\eg / \nu_{k} + \eA);
\quad e_{k}:=\eps_f/\nu_{k} +  \eps_c.
\ee

For the following lemma recall that the constants $\zeta$ and $ \xi$ are defined in lines 2-3 of Algorithm~1.

\begin{lemma}[Increase of the Trust Region] 
\label{tr-lbBO}
Let Assumptions 1 through 3 be satisfied. Suppose that for an iterate $k$ and a given positive constant $\gamma$,
\be {\|\tilde A_k^T \tilde c_k\|} >\mE_v(k) +  \gamma.\label{lem3asp}\ee
Define
\be\label{def_bar_del} \bar\Delta(\gamma) = \left[\frac{ {\pi_1}   \zeta}{\xi  \max(1,M_c) M}\right]\gamma,\ee 
where
\be M = \max\left[\frac{L_f + M_W}{\nu_0},\ L_c, M_A^2 \right].\ee
Then,
 \be  \min\left(\bar\Delta(\gamma),\frac{\|\tA_k^T \tc_k\|}{\|\tA_k^T\tA_k\|}\right)  = \bar\Delta(\gamma). \label{jegg} \ee
Furthermore, if $\Delta_k \leq \bar\Delta(\gamma)$, the step is accepted and
\be \Delta_{k+1} =\tau \Delta_k. \label{dincrease} \ee



\end{lemma}


\begin{proof}
\textit{Part 1.} By \eqref{def_ML}, and since  $\nu_k$ is non-decreasing we obtain:
\be \frac{M_L(\nu_k)}{\nu_k} \leq \max\left[\frac{L_f + M_W}{\nu_0},\ L_c \right] \leq \max\left[\frac{L_f + M_W}{\nu_0},\ L_c, M_A^2 \right]= M.\label{ax1_tr}\ee
By condition \eqref{lem3asp} and the bound of $\|\tc_k\|$ in \cref{constants}, 
\be \frac{\|\tilde A_k^T \tilde c_k\|}{\|\tilde c_k\|} >\frac{\xi }{{\pi_1} \zeta} \left(\frac{\eg}{\nu_k} +  \eA\right)+  \frac{\gamma}{M_c}. \label{lem3_aux2}\ee
Now, by the definitions  of $\bar\Delta(\gamma)$  and $\xi$, 
\be\label{min_delta_bar}
\begin{split}
 \bar\Delta(\gamma)
 &\leq  \frac{ {\pi_1}   \zeta  (1-\pi_0)}{2  \max(1,M_c) M}\gamma\\
 &<  \frac{1}{  \max(1,M_c) M}\gamma\\
 & \leq \frac{\gamma}{M_A^2}\\
 &< \frac{\|\tA_k^T \tc_k\|}{\|\tA_k^T\tA_k\|},
 \end{split}
 \ee
 where the second inequality follows  by noting that $\frac12 {\pi_1}\zeta(1-\pi_0)<1$; the third inequality follows from $\max(1,M_c)\geq 1$ and  $M \geq M_A^2$, by definition; and the last inequality follows from \eqref{lem3asp} and the definition of $M_A$.
 Therefore we have
 \be  \min\left(\bar\Delta(\gamma),\frac{\|\tA_k^T \tc_k\|}{\|\tA_k^T\tA_k\|}\right)  = \bar\Delta(\gamma).  \ee
 
 \textit{Part 2.} Now, since  $\Delta_k \leq \bar\Delta(\gamma)$ and {by $\zeta < 1$,} we have
  \be \min\left(\zeta\Delta_k,\frac{\|\tA_k^T \tc_k\|}{\|\tA_k^T\tA_k\|}\right)  = \zeta \Delta_k. \label{smaller_delta}\ee
We also have that
\be
\begin{split} 
M \bar \Delta(\gamma) + (\eg/\nu_k + \eA) &=  \frac{ {\pi_1}   \zeta}{\xi  \max(1,M_c)}\gamma +  \eg/\nu_k + \eA\\
&\leq \frac{{\pi_1}   \zeta }{\xi M_c}\gamma +   \eg/\nu_k + \eA\\
& = \frac{{\pi_1}  \zeta}{\xi}\left[\frac{\gamma}{M_c} + \xi\frac{1}{{\pi_1} \zeta} (\eg/\nu_k + \eA)\right].
\end{split}
\label{lem3_aux}
\ee
Using this bound, the definition of $\rho_k$ along with \cref{ax1_tr,smaller_delta}, we obtain
\be
\begin{split}
|\rho_k-1| &= \frac{|\ared_k(p_k)- \pred_k(p_k)|}{|\pred_k(p_k)+\xi ( \ef + \nu_k \ec)|}\\ 
&\stackrel{\leq}{\mtiny{\labelcref{merit_rule}}}  
	\frac{|\ared_k(p_k)-\pred_k(p_k)|}{{\pi_1} \nu_k|\vpred_k(p_k)|+\xi ( \ef + \nu_k \ec)} \\
&\stackrel{\leq}{ \mtiny{ \labelcref{lem1_res},\labelcref{lem2_res}}}  
	\frac{M_L(\nu_k) \Delta_k^2 + (\eg + \nu_k \eA) \Delta_k + 2(\ef + \nu_k \ec)}
{ {\pi_1} \nu_k   \frac{\|\tA_k^T \tc_k\|}{2\|\tc_k\|}\min\left(\zeta\Delta_k,\frac{\|\tA_k^T \tc_k\|}{\|\tA_k^T\tA_k\|}\right)+\xi ( \ef + \nu_k \ec)}\\
&=
	\frac{[{(M_L(\nu_k)/\nu_k)} \Delta_k + (\eg + \nu_k \eA)/\nu_k] \Delta_k + 2(\ef + \nu_k \ec)/\nu_k}
{ {\pi_1}    \frac{\|\tA_k^T \tc_k\|}{2\|\tc_k\|}\min\left(\zeta\Delta_k,\frac{\|\tA_k^T \tc_k\|}{\|\tA_k^T\tA_k\|}\right)+\xi ( \ef + \nu_k \ec)/\nu_k}\\
&\stackrel{\leq}{\mtiny{ \labelcref{ax1_tr}, \labelcref{lem3_aux2}}} 
	\frac{[M \Delta_k + (\eg/\nu_k + \eA)] \Delta_k + 2(\ef/\nu_k +  \ec)}{{\pi_1}    \zeta \{\frac{\xi }{{\pi_1} \zeta} (\eg /\nu_k +  \eA) +  \frac{\gamma}{M_c}\}\Delta_k/2+\xi ( \ef/\nu_k +  \ec)}\\
&\stackrel{\leq}{\mtiny{\Delta_k \leq\bar\Delta}}
	\frac{[M \bar\Delta + (\eg/\nu_k + \eA)] \Delta_k + 2(\ef/\nu_k +  \ec)}{{\pi_1}    \zeta \{\frac{\xi }{{\pi_1} \zeta} (\eg /\nu_k +  \eA) +  \frac{\gamma}{M_c}\}\Delta_k/2+\xi ( \ef/\nu_k +  \ec)}\\
&\stackrel{\leq}{\mtiny{\labelcref{lem3_aux}}}
	\frac{\frac{{\pi_1}  \zeta}{\xi}\left[\frac{\gamma}{M_c} + \frac{\xi}{{\pi_1} \zeta} (\eg/\nu_k + \eA)\right] \Delta_k + 2(\ef/\nu_k +  \ec)}{{\pi_1}    \zeta \{\frac{\xi }{{\pi_1} \zeta} (\eg/ \nu_k + \eA) +  \frac{\gamma}{M_c}\}\Delta_k/2+\xi ( \ef/\nu_k +  \ec)}\\
&=
	\frac
{\frac{1}{\xi}\left[\frac{{\pi_1}  \zeta\gamma}{M_c} + \xi (\eg/\nu_k + \eA)\right] \Delta_k + 2(\ef/\nu_k +  \ec)}
{\frac12 \left[\xi(\eg/ \nu_k+ \eA) +  \frac{{\pi_1}    \zeta \gamma }{M_c}\right]\Delta_k+\xi ( \ef/\nu_k +  \ec)}\\
&=\frac{2}{\xi}\\
&=1-\pi_0.
\end{split}\ee
By line 17 of {Algorithm~1 we conclude that \eqref{dincrease} holds.}
\end{proof} 



\begin{corollary}[Lower Bound of Trust Region Radius]\label{cor1}
Let Assumptions~1 through 3 be satisfied. Given $\gamma>0$, if there exist $ K>0$ such that for all $k\geq K$
\be
{\|\tilde A_k^T \tilde c_k\|} > \mE_v(k) +  \gamma,\label{tr-ctaspBO}
\ee
then there exist $\hat K \geq K $ such that for all $k\geq \hat K$,
\be \Delta_k > \tfrac{1}{\tau}\bar\Delta (\gamma).\label{tr-bardelBO} \ee
\end{corollary}

\begin{proof}
We apply \cref{tr-lbBO} for each iterate after $K$ to deduce that, whenever $\Delta_k \leq \, \bar \Delta(\gamma)$, the trust region radius will be increased. Thus, there is an index $\hat K$ for which $\Delta_k$ becomes greater than $\bar\Delta(\gamma)$. On subsequent iterations, the trust region radius can never be reduced below $\bar\Delta(\gamma)/\nu$ (by Step~6 of Algorithm~\ref{algorithmBO}) establishing \eqref{tr-bardelBO}. 
\end{proof}


Before presenting the next lemma, we define several constants that will be useful in the rest of this section. First, we define
\be \chi:= \frac{\pi_0 {\pi_1}^2   \zeta^2}{2\tau\xi M_c \max(1,M_c) M}.\ee
Next, for any given iterate $k'$, recall as first defined in \eqref{defi1}, 
\be \label{defi1'} \mE_v(k'):= \frac{\xi M_c}{{\pi_1} \zeta}  (\eg / \nu_{k'} + \eA);
\quad e_{k'}:=\eps_f/\nu_{k'} +  \eps_c
\ee
Additionally, \textit{for any given} $\mu>0$, define
\be \label{defi2}
\gamma_{k'} :=  \frac12 \left(-\mE_v(k') + \sqrt{\mE_v(k')^2 + 8 e_{k'} / \chi} \right)+ \mu;
\quad\bar\Delta_{k'} = \frac{ {\pi_1}   \zeta}{\xi  \max(1,M_c) M}\gamma_{k'} .
\ee
Thus, here and henceforth we write $$\bar\Delta_{k'}:= \bar \Delta (\gamma_{k'}).$$
Note that the four quantities defined in \eqref{defi1'}-\eqref{defi2} only depend on $k'$ through the value of the penalty parameter $\nu_{k'}$.

\medskip\noindent
\textit{Remark 1. The Anchor Iterate $k'$.} We emphasize that $k'$ denotes an arbitrary positive integer. All subsequent results  will be presented with respect to this fixed number (and thus on its corresponding merit parameter $\nu_{k'}$). {We call $k'$ the \textit{anchor iterate},} and revisit its role later on after introducing the first two critical regions in propositions~1 and 2.  


For convenience, we also introduce a re-scaled version of the merit function, 
\be \label{scaled} \tilde \Phi(x,\nu):= \tfrac{1}{\nu} \tf(x) + \|\tc(x)\|, \ee
as well as its noiseless counterpart,
\be \label{nonscale} \Phi(x,\nu):= \tfrac{1}{\nu} f(x) + \|c(x)\|. \ee

With these definitions at hand, we are ready to state our next lemma.
\begin{lemma}[Merit Function Reduction]\label{noisyfuncredBO}
Let Assumptions 1 through 3 be satisfied. Let $k'$ be any non-negative integer and let $\mu>0$ {in \eqref{defi2}} be any fixed constant. Suppose for some iterate $k>k'$,
\be \label{suppBO}
    {\|\tilde A_k^T \tilde c_k\|} > \mE_v(k') + \gamma_{k'}
   \quad\mbox{and}\quad
   \Delta_k \, {\geq} \, \frac{\bar\Delta_{k'}}{\tau}.
 \ee 
Then
\be 
\vpred_k(p_k) 
\geq \frac{\chi}{\pi_0 {\pi_1}}\left(\mE_v(k')+  \gamma_{k'}\right)\gamma_{k'}.
\ee
Furthermore, if the step is accepted at iteration $k$ {by Algorithm~\ref{algorithmBO}}, we have
\be
\tilde \Phi({x_{k}} , \nu_k) - \tilde \Phi({x_{k+1}} , \nu_k) > \chi \mu^2 +\mu \sqrt{\chi^2\mE_v(k')^2 + 8\chi e_{k'}} \label{meritdec}.
\ee

\end{lemma}

\begin{proof} 

We first note that since $\nu_k$ can only be increased throughout the optimization process, 
\be \label{comp_errs}
\mE_v(k') \geq \mE_v(k);\quad e_{k'} \geq e_k. 
\ee
{Combining this fact with \eqref{suppBO}, we
have:
\be {\|\tilde A_k^T \tilde c_k\|} >\frac{\xi M_c}{{\pi_1} \zeta}  (\eg / \nu_k + \eA) +  \gamma.\ee

note that condition \eqref{lem3asp} in Lemma~\ref{tr-lbBO} holds, as we take 
$\gamma = \gamma_{k'}.$ Consequently, part 1 of the proof of Lemma~\ref{tr-lbBO} applies and we have that \eqref{jegg} is satisfied.} 
It follows that

\be\begin{split}
\label{tr-lem36_1BO}
\min \left(\zeta\Delta_k,\frac{\|\tA_k^T \tc_k\|}{\|\tA_k^T\tA_k\|}\right)
&\stackrel{\geq}{\mtiny{\labelcref{suppBO}}} \min\left(\frac{\zeta}{\tau}\bar\Delta_{k'},\frac{\|\tA_k^T \tc_k\|}{\|\tA_k^T\tA_k\|}\right)\\
&{\stackrel{\geq}{\mtiny{\labelcref{jegg}}} \frac{\zeta}{\tau} \bar\Delta_{k'}} \\ 
&=\frac{ {\pi_1}   \zeta^2}{\tau\xi  \max(1,M_c) M}\gamma_{k'}.
\end{split}
\ee
By \eqref{lem1_res}, 
\be 
\label{vpred_gamma}
\begin{split}
\vpred_k(p_k) 
    &\geq \frac{\|\tA_k^T \tc_k\|}{2\|\tc_k\|}\min\left(\zeta\Delta_k,\frac{\|\tA_k^T \tc_k\|}{\|\tA_k^T\tA_k\|}\right)\\
    & \stackrel{\geq}{\mtiny{\labelcref{suppBO}\labelcref{tr-lem36_1BO}}}
    	\frac1{2\|\tilde c_k\|} \left(\frac{\xi M_c}{{\pi_1} \zeta}  (\eg / \nu_{k'} + \eA) +  \gamma_{k'}\right)\frac{ {\pi_1}   \zeta^2}{\tau\xi  \max(1,M_c) M}\gamma_{k'}\\
    & \stackrel{\geq}{\mtiny{\labelcref{constants}}}
    	\frac12 \left(\frac{\xi}{{\pi_1} \zeta}  (\eg / {\nu_{k'}} + \eA) +  \frac{\gamma_{k'}}{M_c}\right)\frac{ {\pi_1}   \zeta^2}{\tau\xi  \max(1,M_c) M}\gamma_{k'}\\
    & = 
    	\frac{ {\pi_1}   \zeta^2}{2\tau\xi M_c \max(1,M_c) M} \left(\frac{\xi M_c}{{\pi_1} \zeta}  (\eg / {\nu_{k'}} + \eA) +  \gamma_{k'}\right)\gamma_{k'}\\
    & = 
    	\frac{\chi}{\pi_0 {\pi_1}}\left(\mE_v(k')+  \gamma_{k'}\right)\gamma_{k'}.
\end{split}
\ee
This proves the first part of the lemma.

Let the step $p_k$ be accepted. Then by line~16  of the Algorithm~1 and definition \eqref{rhodef} of $\rho_k$ and definition of $\xi$ in line 3 of the Algorithm,
\be \label{redex}
    \ared_k > \pi_0 \pred_k + (\pi_0 -1)\xi(\eps_f + \nu_k \eps_c) = \pi_0 \pred_k - 2(\eps_f + \nu_k \eps_c) .
\ee
Recalling the definition of $\ared_k$ and  condition \eqref{merit_rule}  
\be \label{redex1}
\tilde \phi(x_k , \nu_k) - \tilde \phi(x_k + p_k , \nu_k) > \pi_0 {\pi_1} \nu_k \vpred_k - 2(\eps_f + \nu_{k} \eps_c).
\ee
Dividing through by $\nu_k$, and using the relationship $e_{k'}\geq e_k$ we obtain
\be \begin{split}
\tilde \Phi(x_k , \nu_k) - \tilde \Phi(x_k + p_k , \nu_k) &> \pi_0 {\pi_1} \vpred_k - 2 e_k  \\ 
		& \geq   \pi_0 {\pi_1} \vpred_k - 2 e_{k'} .
\end{split} \label{redex2}\ee
We use  \eqref{vpred_gamma} 
to obtain
\be 
\begin{split} 
    & \tilde \Phi(x_k , \nu_k) - \tilde \Phi(x_k + p_k , \nu_k) \\
    & > \pi_0{\pi_1}\vpred_k - 2 e_{k'}\\
    & = \chi (\mE_v(k')+\gamma_{k'}) \gamma_{k'} - 2 e_{k'}\\
    & = \frac\chi4 \left[2\mE_v(k')+ \left(-\mE_v(k') + \sqrt{\mE_v(k')^2 + 8 e_{k'} / \chi} \right)+2\mu\right] \left(-\mE_v(k') + \sqrt{\mE_v(k')^2 + 8 e_{k'} / \chi } + 2\mu\right)- 2 e_{k'}\\
    & = \frac\chi4 \left[\mE_v(k') + \left(\sqrt{\mE_v(k')^2 + 8 e_{k'} / \chi} +2\mu\right)\right] \left[-\mE_v(k') + \left(\sqrt{\mE_v(k')^2 + 8 e_{k'} / \chi } + 2\mu\right)\right]- 2 e_{k'}\\
    & = \frac\chi4 \left[\left(\sqrt{\mE_v(k')^2 + 8 e_{k'} / \chi } + 2\mu\right)^2-\mE_v(k')^2\right]- 2 e_{k'} \\
    & = \frac\chi4 \left[ 8 e_{k'} / \chi + 4\mu^2 + 4\mu\sqrt{\mE_v(k')^2 + 8 e_{k'} / \chi }\right]-2 e_{k'}\\
    & = \chi \mu^2 +\mu \sqrt{\chi^2\mE_v(k')^2 + 8\chi e_{k'}}.  
\end{split} \label{leftout}
\ee

\end{proof}

\begin{obs}[Monotonicity of Rescaled Merit Function]
\label{monomerit} 
By Assumption~2, {$\{f_k\}$} is bounded below. We may thus redefine the objective function (by adding a constant) so that for all $x_k$, $\tilde f(x_k) > 0$, without affecting  the problem or the algorithm. As a consequence,  for any iterate $x_k$ and merit parameters $\nu_a \geq \nu_b$, the rescaled merit function satisfies  
	\be \label{monoq} \tilde \Phi(x_k , \nu_a) - \tilde \Phi(x_k , \nu_b)\leq 0, \ee 
since
$
		\tilde \Phi(_k , \nu_a) - \tilde \Phi(x_k , \nu_b) = \left( \frac{1}{\nu_a} -\frac{1}{\nu_b} \right) \tilde f(x_k) \leq 0. 
	$
\end{obs}

	

{We can now show that the measure of stationarity for feasibility can be reduced to a level consistent with the noise present in the problem.}

\begin{proposition}[Finite Time Entry to Critical Region I of Feasibility]\label{prop1}
	Suppose that  Assumptions~1 through 3 are satisfied. Let $k'$ {denote the anchor iterate mentioned above}.
	Then, the sequence of iterates $\{x_k\}$ generated by Algorithm~\ref{algorithmBO} visits infinitely often the critical region $C_{Ac}^{I}(k')$ be defined as
     \be C_{Ac}^{I}(k') =\left\{x: \| A(x)^T  c(x)\| \leq  \mE_v(k') + \eps_A M_c + \eps_c M_A + \eps_A\eps_c+\gamma_{k'} :=\mathcal{E}_{Ac}^I. \right\} 
    \label{EAcI}\ee
    (We write  $\mathcal{E}_{Ac}^I$instead of $\mathcal{E}_{Ac}^I(k')$ for ease of notation).
\end{proposition}

\begin{proof}
	We proceed by  means of contradiction. Assume  that there exist an integer {$K> k'$}, such that for all $k>K$, none of the iterates is contained in $C_{Ac}^{I}(k')$, i.e.
    \be \| A(x_k)^T  c(x_k)\| >  \mE_v(k') + \eps_A M_c + \eps_c M_A + \eps_A\eps_c+\gamma_{k'}.\label{AcRelation}\ee
    Therefore, for all $k>K$,
    \be\label{tAc}
        \begin{split}
            &\| \tA(x_k)^T  \tc(x_k)\|\\
            &= \| [A(x_k)+\delta_A(x_k)]^T  [c(x_k)+\delta_c(x_k)]\|\\
            &\geq \| A(x_k)^T  c(x_k)\| - \| A(x_k)^T\delta_c(x_k)\| - \|\delta_A(x_k)^T c(x_k)\| - \|\delta_A(x_k)^T\delta_c(x_k)\|\\
            &\stackrel{>}{\mtiny{\labelcref{AcRelation}}}\mE_v(k') + \eps_A M_c + \eps_c M_A + \eps_A\eps_c+\gamma_{k'} -  (\eps_A M_c + \eps_c M_A + \eps_A\eps_c)\\
            &=\mE_v(k') + \gamma_{k'}\\
            &\stackrel{\geq}{\mtiny{\labelcref{comp_errs}}} \mE_v(k) + \gamma_{k'}.
        \end{split}
    \ee
    Therefore \cref{cor1} applies with $\gamma = \gamma_{k'}$,
    implying that there is an index $\hat K$ such that for $k = \hat K, \hat K+1, ... $, we have 
    \be \label{prop1_delta_bd}
    \Delta_k>\frac1\tau\bar\Delta_{k'}.
    \ee 
    We then apply \cref{noisyfuncredBO} for $k = \hat K, \hat K+1, ... $, to conclude that all accepted steps satisfy
    \be
     \tilde \Phi(x_k , \nu_k) - \tilde \Phi(x_{k+1} , \nu_k) > \chi \mu^2 +\mu \sqrt{\chi^2\mE_v(k')^2 + 8\chi e_{k'}.}
    \ee 
    
     Furthermore, there are infinitely many accepted steps after $\hat K$, since otherwise there exists an iterate $\hat K'$ such that for all iterates $k\geq K'$ the steps are rejected, and by line 19 of the Algorithm~1 we would have that $\Delta_k\rightarrow 0$ as $k\rightarrow \infty$, contradicting \eqref{prop1_delta_bd}.
	
	Therefore, we focus on the iterates after $\hat K$ for which the step is accepted. They
 form a subsequence $\{x_{k_j}\}$, for $j = 1, 2, ...$. We note that for any $j$, 
	\be \|\tilde A(x_{k_j})^T \tilde c(x_{k_j})\| >  \mE_v(k) + \gamma_{k'},\quad \Delta_{k_{j}} > \frac{\bar \Delta_{k'}}{\tau}.\ee 
	By \eqref{monoq} and \eqref{meritdec},
	\be 
	\begin{split}
	\tilde\Phi(x_{k_j} , \nu_{k_j}) -\tilde\Phi(x_{k_j+1} , \nu_{k_j+1})  
	= & ~ \tilde\Phi(x_{k_j} , \nu_{k_j}) -\tilde\Phi(x_{k_j+1} , \nu_{k_j}) + \tilde\Phi(x_{k_j+1} , \nu_{k_j}) -\tilde\Phi(x_{k_j+1} , \nu_{k_j+1})\\
	\geq &~ \tilde\Phi(x_{k_j} , \nu_{k_j}) -\tilde\Phi(x_{k_j+1} , \nu_{k_j})\\
	\geq &~ \chi \mu^2 +\mu \sqrt{\chi^2\mE_v(k')^2 + 8\chi e_{k'}}.
	\end{split}
	\ee
	
	Since there are infinitely many accepted steps, this implies that $\{\tilde\Phi(x_{k_j}, \nu_{k_j} \}$ is unbounded below, which is not possible since $\{\tilde f_k\}$ is bounded below by Assumption 2. This contradiction completes the proof.
\end{proof}


{This result addresses the scenario in which the Jacobian $ \tA_k $ undergoes a loss of rank. Specifically, we show that $ \|\tA^T \tc \| $ falls below a noise-scaled threshold in every case. Similar to the classical setting, the smallness of $ \|\tA^T \tc\| $ may indicate that $ \tA $ is nearing singularity.  Furthermore, in \cref{cor2} we establish that if $ \tA $ stays sufficiently far from singularity, then $\| \tc \|$ decreases below a noise-scaled threshold.}

The following lemma helps measure how far can the iterates  stray away from the region  $C_{Ac}^{I}(k')$, after exiting this region and before returning
 to it. 

\begin{lemma}[Displacement Bound Outside of Critical Region I]\label{disp_v}
Let Assumptions 1 through 3 be satisfied and let $k'$ be the anchor iterate used in the previous results. Let $k_1 > k'$ be such that $x_{k_1} \in C_{Ac}^{I}(k')$  and $x_{k_{1}+1}\notin C_{Ac}^{I}(k')$.
Then, if $\Delta_{k_1} < \bar\Delta_{k'}$,  there exist a finite iterate $k_2\geq k_{1}+1$, defined as
\be \label{k2}
k_2 = \min \left\{ k \geq k_1 + 1 : \Delta_k \geq \bar{\Delta}_{k'} \text{ or } x_k \in C_{Ac}^I(k') \right\}.
\ee
Furthermore, for any $k$ with $k_1 \leq k \leq k_2$ we have that
\be \| x_{k} - x_{k_1} \| \leq \frac{\tau}{\tau-1} \bar\Delta_{k'} \ee
\end{lemma}

\begin{proof}
We  show the first part of the lemma by  means of contradiction. Assume for contradiction that $k_2$ is not finite. Therefore, for $k = k_1+1, k_1+2 , ... $, 
\be \Delta_k < \bar\Delta_{k'}\label{lm5_delta_bd}\ee 
and
\be x_k\notin C_{Ac}^{I}(k'),\ee 
which as argued in \eqref{tAc}, implies
\be  \|\tA_k^T \tc_k\|\geq \mE_v(k) +\gamma_{k'}. \ee
Therefore we apply \cref{tr-lbBO} for each iterate $k \geq k_1+1$ and obtain that $\Delta_k\rightarrow \infty$ as $k\rightarrow \infty$, contradicting \eqref{lm5_delta_bd}.

For the rest of the lemma, we take any $k$ with $k_1 < k < k_2$ and have that $x_k \notin C_{Ac}^{I}(k')$,
and thus again as argued in \eqref{tAc},
\be\|\tilde A(x_k)^T \tilde c(x_k)\| >  \mE_v(k) + \gamma_{k'}. \ee
By assumption, each of the iterates $k= k_1 , ... , k_2-1$ satisfy 
$\Delta_{k} < \bar\Delta_{k'}.$ Therefore by \cref{tr-lbBO}, 
$ \Delta_{k+1} = \tau \Delta_k $, and thus for $i = 0, 1, ..., k_2-k_1-1$
\be \Delta_{k_2-1-i} = \tau^{-i} \Delta_{k_2-1} < \tau^{-i} \bar \Delta_{k'}.\ee
It follows that
\begin{align*}
    \|x_{k} - x_{k_1} \| &\leq \sum_{i = 1}^{k - k_1} \|x_{k_1 + i} - x_{k_1 + i-1}\|
				\leq \sum_{i = 1}^{k_2 - k_1} \|x_{k_1 + i} - x_{k_1 + i-1}\|\\
				&\leq \sum_{j = k_1}^{k_2-1} \Delta_{j}
			         = \sum_{i = 0}^{ k_2 - k_1 - 1} \tau^{-i} \Delta_{k_2-1}\\
    				& < \bar \Delta_{k'}\sum_{i = 0}^{\infty}\tau^{-i}
    				 = \frac{\tau}{\tau-1} \bar \Delta_{k'},
\end{align*}
which concludes the proof.
\end{proof}

We now define the maximum value of the re-scaled, noiseless merit function $\Phi(x,\nu)$ (defined in \eqref{nonscale}) in $C_{Ac}^{I}(k')$:  
\be \bar\Phi_{Ac}^I(k') = \sup_{x\in C_{Ac}^{I}(k') ,\\ \nu \geq \nu_{k'}}  \Phi(x , \nu) \label{PhiIAc}.\ee
Similarly, we define 
\be \bar G^I_{Ac}(k') = \sup_{x\in C_{Ac}^{I}(k')} {\|g(x)\|}.\ee

\begin{proposition}[Remaining in Critical Region II of Feasibility]\label{prop2}
Once an iterate enters  $C_{Ac}^{I}(k')$, {the sequence $\{x_k\}$} never leaves the set $C_{Ac}^{II}(k')$ defined as
\be
 C_{Ac}^{II}(k') =\left\{x: \Phi( x , \nu )  \leq  \bar\Phi_{Ac}^I(k') +  \max(\mathcal{P}_{Ac}^{II}{(k')} , 2 e_{k'}  )+ 2 e_{k'}  := E_{Ac}^{II}\right\},
\label{EAcII}
\ee
where {$\Phi$ is defined in \eqref{nonscale}} and
\be \mathcal{P}_{Ac}^{II}{(k')}= \left[    \frac{\bar G^I_{Ac}(k')}{\nu_{k'} }   +\EAcI(k') +      \frac{ {\pi_1}  \tau \zeta (  L_f/\nu_{k'} + L_c )}{\xi (\tau-1) \max(1,M_c) M}\gamma_{k'}  \right] \frac{ {\pi_1}  \tau \zeta}{\xi (\tau-1) \max(1,M_c) M}\gamma_{k'}.\ee
\end{proposition}

\begin{proof}
We let $k_1$ and $k_2$  be defined as in the last lemma:
\be x_{k_1}\in C_{Ac}^{I}(k'), \quad x_{k_1+1}\notin C_{Ac}^{I}(k'),\ee
\be \label{k2_prop2}
k_2 = \min \left\{ k \geq k_1 + 1 : \Delta_k \geq \bar{\Delta}_{k'} \text{ or } x_k \in C_{Ac}^I(k') \right\},
\ee
and recall that $k_2$ is finite.

Since we consider only iterates $k$ with  $k\geq k'$, we have for $k = k_1, .... $
\be \label{phi_dist}
\begin{split}
    |\tilde \Phi(x_{k},\nu_k) - \Phi(x_{k},\nu_k) |
      \leq& | \delta_f(x_{k})/\nu_{k}| + \| \delta_c(x_{k}) \|\\
    \leq& \frac{\eps_f}{\nu_{k}} + \eps_c\\
    \leq& \frac{\eps_f}{\nu_{k'}} + \eps_c\\
     =& e_{k'},
\end{split}
\ee
where the last inequality follows from \eqref{defi1}.
Since the step from $k_1 $ is accepted, we have {that \eqref{redex}-\eqref{redex2} hold for $k= k_1$ and thus
\be \tilde\Phi(x_{k_1},\nu_{k_1}) - \tilde\Phi(x_{{k_1+1}},\nu_{{k_1}}) >  -2e_{k'}, \ee
}
By the monotonicity result 
\cref{monoq} we have that $\tilde\Phi(x_{k_1},\nu_{k_1}) - \tilde\Phi(x_{{k_1+1}},\nu_{k_1+1}) \geq \tilde\Phi(x_{k_1},\nu_{k_1}) - \tilde\Phi(x_{{k_1+1}},\nu_{k_1})$, and thus
\be \tilde\Phi(x_{k_1},\nu_{k_1}) - \tilde\Phi(x_{{k_1+1}},\nu_{{k_1+1}}) >   {-2e_{k'}}. \ee
Recalling definition \eqref{PhiIAc} and the fact that the $k_1$ iterate is in $C_{Ac}^I(k')$, we have
\be 
\tilde\Phi(x_{{k_1+1}},\nu_{{k_1+1}}) 
<  \tilde\Phi(x_{k_1},\nu_{k_1}) + {2e_{k'}}
\stackrel{<}{\mtiny{\labelcref{phi_dist}}} \Phi(x_{k_1},\nu_{k_1}) + {3e_{k'}}
\leq \bar\Phi_{Ac}^I(k')+ {3e_{k'}}. \label{exit_bd}\ee

{We divide the rest of the proof into two cases based
on whether  $\Delta_{k_1} \geq \bar\Delta_{k'}$ or not.

\medskip
{\bf Assume} $\Delta_{k_1} \geq {\bar{\Delta}_{k'}}$. For each $k = k_1+1, \ldots, k_2-1$, it follows that $x_k \notin C_{Ac}^I(k')$. According to \cref{tAc}, this implies $\|\tA_k^T \tc_k\| \geq \mE_v(k) + \gamma_{k'}$, so that condition \cref{lem3asp} in \cref{tr-lbBO} is satisfied. Now, for $k \in \{ k_1+1, \ldots, k_2-1\}$ the trust region radius can decrease, but by 
 \cref{tr-lbBO}, if at some point $\Delta_k < \bar\Delta_{k'}$ then $\Delta_{k+1} = \tau \Delta_k$. We deduce that $\Delta_k > \frac{\bar{\Delta}_{k'}}{\tau}$ for all  $k \in \{ k_1+1, \ldots, k_2-1\}$.
 We then apply \cref{noisyfuncredBO} 
 to conclude that each accepted step  reduces the merit function from $\tilde\Phi(x_{{k_1+1}},\nu_{{k_1+1}})$, so that by \cref{monoq} we have that for each step $k$ after the exiting iterate $k_1+1$,
\be\label{prop1_key1}
\Phi(x_{{k}},\nu_{{k}}) \leq \tilde\Phi(x_{{k}},\nu_{{k}}) + e_{k'} <\tilde\Phi(x_{{k_1+1}},\nu_{{k_1+1}})+ e_{k'} \stackrel{<}{\mtiny{\labelcref{exit_bd}}} \bar\Phi_{Ac}^I(k')+ 4e_{k'} \leq E_{Ac}^{II}. 
\ee
This concludes the proof for when $\Delta_{k_1} \geq {\bar{\Delta}_{k'}}$.
}

\medskip {\bf Assume} $\Delta_{k_1} < \bar\Delta_{k'}$. Using \cref{disp_v}, we can bound the displacement of iterates from $k_1$ to any $k = k_1 +1, \ldots, k_2$. Specifically, by \cref{disp_v}, for $k_1 \leq k \leq k_2$
\be \| x_{k} - x_{k_1} \| \leq \frac{\tau}{\tau-1} \bar\Delta_{k'}. \ee
By $L_f$ and $L_c$--Lipschitz differentiability of the objective and the constraints, respectively, we have for any $k= k_1 , ... , k_2$:
\be 
\begin{split}
f(x_k) - f(x_{k_1}) & \leq \max_{t\in[0,1]}\|g(t x_{k_1} +(1-t) x_k ) \| \|x_k - x_{k_1}\|\\
		      & \leq \left[\|g(x_{k_1})\| + L_f  \| \|x_k - x_{k_1}\|\right] \| \|x_k - x_{k_1}\|\\
		      & \leq \left[\bar G^I_{Ac}(k') + \frac{\tau L_f}{\tau-1}\bar\Delta_{k'}\right] \frac{\tau}{\tau-1}\bar\Delta_{k'}.
\end{split}
\ee
Similarly, for any $k_1 \leq k \leq k_2$, 
\be 
\begin{split}
\|c(x_k)\| - \| c(x_{k_1})\| & \leq \max_{t\in[0,1]}\|\grad c(t x_{k_1} +(1-t) x_k ) \| \|x_k - x_{k_1}\|\\
		      & \leq \left[\| A^T(x_{k_1})c(x_{k_1})\| + L_c  \| \|x_k - x_{k_1}\|\right] \| \|x_k - x_{k_1}\|\\
		      & \leq \left[\EAcI(k') + \frac{\tau L_c}{\tau-1}\bar\Delta_{k'}\right] \frac{\tau}{\tau-1}\bar\Delta_{k'}.
\end{split}
\ee
Using these two last results and recalling the definition \eqref{defi2} of $\bar \Delta_{k'}$ we find, for any $k_1 \leq k \leq k_2$,
\be \label{prop1_res1}
\begin{split}
\Phi(x_k,\nu_k) - \Phi(x_{k_1},\nu_{k_1})
	& =  \frac1{\nu_k} f(x_k) - \frac1{\nu_{k_1}} f(x_{k_1}) + \|c(x_k)\| - \| c(x_{k_1})\| \\
	& \leq \frac1{\nu_{k_1}} [f(x_k) - f(x_{k_1})]+ \|c(x_k)\| - \| c(x_{k_1})\| \\
	& \leq \frac1{\nu_{k_1}} \left[\bar G^I_{Ac}(k') + \frac{\tau L_f}{\tau-1}\bar\Delta_{k'}\right] \frac{\tau}{\tau-1}\bar\Delta_{k'}+\left[\EAcI + \frac{\tau L_c}{\tau-1}\bar\Delta_{k'}\right] \frac{\tau}{\tau-1}\bar\Delta_{k'}\\
	& \leq \frac1{\nu_{k'}} \left[\bar G^I_{Ac}(k') + \frac{\tau L_f}{\tau-1}\bar\Delta_{k'}\right] \frac{\tau}{\tau-1}\bar\Delta_{k'}+\left[\EAcI + \frac{\tau L_c}{\tau-1}\bar\Delta_{k'}\right] \frac{\tau}{\tau-1}\bar\Delta_{k'}\\
	& =   \left[    \frac{\bar G^I_{Ac}(k')}{\nu_{k'} }   + \EAcI +     \left(  \frac{L_f}{\nu_{k'} } + L_c\right) \frac{\tau \bar\Delta_{k'}}{\tau-1}  \right] \frac{\tau \bar\Delta_{k'}}{\tau-1}\\
	& =  \left[    \frac{\bar G^I_{Ac}(k')}{\nu_{k'} }   +\EAcI +      \frac{ {\pi_1}  \tau \zeta (  L_f/\nu_{k'} + L_c )}{\xi (\tau-1) \max(1,M_c) M}\gamma_{k'}  \right] \frac{ {\pi_1}  \tau \zeta}{\xi (\tau-1) \max(1,M_c) M}\gamma_{k'}\\
	&= \mathcal{P}_{Ac}^{II}(k').
\end{split}
\ee
Therefore we find for any $k_1 \leq k \leq k_2$,
\be\label{prop1_key2}
\begin{split}
    \Phi(x_k,\nu_k) &\leq \Phi(x_{k_1},\nu_{k_1}) + \mathcal{P}_{Ac}^{II} \\
                    & \stackrel{\leq}{\mtiny{\labelcref{phi_dist}}}  \tilde\Phi(x_{k_1},\nu_{k_1}) + \mathcal{P}_{Ac}^{II} + e_{k'} \\
                    & \stackrel{\leq}{\mtiny{\labelcref{exit_bd}}} \bar\Phi_{Ac}^I(k')+ \mathcal{P}_{Ac}^{II}+ 4e_{k'},\\
                    &\leq E_{Ac}^{II}(k').
\end{split}
\ee
If $x_{k_2}\in C_{Ac}^{I}(k')$, the proof is complete. 

On the other hand, if  $x_{k_2}\notin C_{Ac}^{I}(k')$, we only need to show that  \eqref{prop1_key2} is also satisfied by $k=k_2+1, ... ,\hat K$, where 
\be\hat K = \min\{k\geq k_2+1: x_{k}\in C_{Ac}^{I}(k')  \}.\ee
The existence of $\hat K$ is guaranteed by \cref{prop1}.


Setting $k = k_2$ in \eqref{prop1_res1} we get
\be
\Phi(x_{k_2},\nu_{k_2})  \leq \Phi(x_{k_1},\nu_{k_1})+  \mathcal{P}_{Ac}^{II},
\ee
which together with \eqref{phi_dist}  gives
\be\label{phi_k2}
\tilde \Phi(x_{k_2},\nu_{k_2})  \leq \Phi(x_{k_1},\nu_{k_1})+  \mathcal{P}_{Ac}^{II} + e_{k'}\leq \bar\Phi_{Ac}^I(k')+  \mathcal{P}_{Ac}^{II} + e_{k'},
\ee
where the last inequality is due to the fact that $k_1\in C_{Ac}^{I}(k')$.
Since that iterates have not yet returned into $C_{Ac}^{I}(k')$ at iterate $k_2$, we apply \cref{noisyfuncredBO} for each of the iterates after $k_2$ until iterates return to $C_{Ac}^{I}(k')$ again at iterate $\hat K$ (such iterate exist due to \cref{EAcI}) and obtain that 
\be \tilde\Phi(x_{k_2},\nu_{k_2}) > \tilde\Phi(x_{k_2+1},\nu_{k_2+1}) > ... > \tilde\Phi(x_{\hat K},\nu_{\hat K}). \label{prop1_res2}\ee
Recalling again \eqref{phi_dist}, we find that for $k = k_2+1, ... , \hat K$,
\be\label{prop1_key3}
\begin{split}
\Phi(x_{k},\nu_{k}) &\leq \tilde\Phi(x_{k},\nu_{k}) + e_{k'} \\
                    &\stackrel{<}{\mtiny{\labelcref{prop1_res2}}}  \tilde\Phi(x_{k_2},\nu_{k_2}) + e_{k'}\\
                    &\stackrel{\leq}{\mtiny{\labelcref{phi_k2}}} \bar\Phi_{Ac}^I(k')+  \mathcal{P}_{Ac}^{II} + 2e_{k'}
\end{split}
\ee
We now combine results from \cref{prop1_key1,prop1_key2,prop1_key3} and conclude the proof.

\end{proof}

\medskip\textit{Remark 2.} The results in \cref{prop1,prop2} depend on the \textit{anchor iterate} $k'$ and the corresponding merit parameter $\nu_{k'}$. As mentioned in Remark~1 (preceding  \eqref{scaled}), we fix the value of $k'$ throughout the analysis. As evident from \cref{EAcI} and \cref{EAcII}, and the definitions \eqref{defi1}-\eqref{defi2}, the sizes of the critical regions are inversely proportional to the value of $\nu_{k'}$.  This seemingly surprising fact is quite revealing. While the analysis presented above would hold if we fix $k'$  at the outset, say $k'=0$, we maintain this generality to make the results more expressive. For example, we will study later on the effect of the term $k'$ in the case when $\nu_k \rightarrow \infty$—which happens only if $\tA_k$ loses rank c.f. \cite{byrd2000trust}.



 \subsection{Feasibility Under the Full Rank Assumption}

If, during the run of Algorithm~1 the Jacobian remains full rank, 
we can establish a stronger result showing that the feasibility measure $\| c(x)\| $ is small. 

\medskip\noindent
\noindent\textbf{Assumption 4:} The singular values of the Jacobian $\{ \tilde A_k\} $ are bounded below by $\sigma_{\min} > 0$.

\medskip\noindent
The following result  follows readily from  \cref{prop1}. 

\begin{corollary}\label{cor2}
    Let Assumption 1 through 4 be satisfied. Then, the subsequence of iterates contained in  $C_{Ac}^I(k')$  satisfies
    \be 
        \| c(x_k)\| \leq  \frac{\mE_v(k') +\gamma_{k'}}{\sigma_{\min}}+\eps_c.
    \ee
\end{corollary}
\begin{proof}
    As argued in \eqref{tAc}, all iterates outside the set $C_{Ac}^I(k')$ must satisfy
    \be  \| \tA_k^T  \tc_k\| \geq \mE_v(k') + \gamma_{k'}, \ee
    and therefore for the infinite sequence of iterates in $C_{Ac}^I(k')$ we have
    \be  \| \tA_k^T  \tc_k\| < \mE_v(k') + \gamma_{k'}. \ee
    By Assumption~4, $\|\tA_k\|\geq \sigma_{\min}$,  and thus 
    \be  \|  \tc_k\| < \frac{\mE_v(k) + \gamma_{k'}}{\sigma_{\min}}. \ee
    We conclude the proof by recalling \eqref{noisebBO}.
\end{proof}

\subsection{Reduction in the Optimality Measure}

{We now study the contribution of the tangential step defined in subproblem \eqref{hprob_exact}. Having computed the normal step $\nu_k$, we write the total step of the algorithm as $p_k = v_k + h_k$, where $h_k$ is to be determined. As already mentioned $v_k$ is in the range space of $\tA_k$, so we require that $h_k$ be in the null space of $\tA_k^T$. Substituting $p_k= v_k +h$ in \eqref{hprob_exact} and ignoring constant terms involving $v_k$, we define obtain the following subproblem:} 

\begin{align}\label{hprob_h}
\min_h\quad  &(\tg_{k}+ \tW_k v_k)^{T} h+\frac{1}{2} h^{T}  \tW_k h\\
 \text{ subject to }\quad &\|h \|\leq \sqrt{\Delta_k^2 - \|v_k\|^2},
\end{align}
where the last inequality follows from the orthogonality of $h$ and $v_k$.
Let $\tZ_k$ be an orthonormal basis
for the null space of $\tA_k^T$. Thus 
\be h_k = \tZ_k d_k, \ee
for some vector $d_k$,
and we can rewrite \eqref{hprob_h} as the reduced tangential problem:

\begin{align}
\min_d \quad  &(\tg_{k}+ \tW_k v_k)^{T} \tZ_k d+\frac{1}{2} d^{T} \left[\tZ_k^T \tW_k \tZ_k\right] d \label{hprob}\\
 \text{ subject to }\quad &\|\tZ_k d \|\leq \sqrt{\Delta_k^2 - \|v_k\|^2} .\nonumber
\end{align}
{In summary, the full step of the algorithm is expressed as}
\benn
    p_k = v_k +\tZ_k d_k
        = v_k + h_k.
\eenn



{To commence the analysis of $h_k$}, we  define the tangential predicted reduction $\hpred_k$ produced by the step $h_k= \tZ_k d_k$ as the change in the objective function in 
\eqref{hprob}

\be \hpred_k(p_k) =-(\tg_{k}+\tW_k v_k)^{T} \tZ_k d_k-\frac{1}{2} d_k^{T} \left[\tZ_k^T \tW_k \tZ_k\right] d_k \ee 
\be = -(\tg_{k}+\tW_k v_k)^{T} h_k-\frac{1}{2} h_k^{T} \tW_k h_k .\ee

Having defined $\hpred_k$, $\pred_k$ and $\vpred_k$, we have from \eqref{model}-\eqref{vpred}
\be \label{pred_vpred_hpred}
\begin{split}
    \pred_k 
    & = m_k(0)- m_k(p_k)\\
    &=  -p_k^{T} \tilde g_{k}-\tfrac{1}{2} p_k^{T} \tilde W_{k} p_k+\nu_k\left( \|\tilde c_k\|-\left\|\tilde A_{k}^{T} p_k+\tilde c_{k}\right\|\right) \\
    &= -(v_k+h_k)^{T} \tilde g_{k}-\tfrac{1}{2} (v_k+h_k)^{T} \tilde W_{k} (v_k+h_k)+\nu_k \vpred_k\\
    &= \nu_k \vpred_k + \hpred_k - \tilde g_k^T v_k - \frac12 v_k^T \tilde W_k v_k.
\end{split}
\ee
It follows from  \eqref{vprobtr}  that
\be \label{passing} \sqrt{\Delta_k^2 - \|v_k\|^2} \geq (1-\zeta)\Delta_k.\ee
Applying the  Cauchy decrease condition \cite{mybook,conn2000trust}  to problem \eqref{hprob}, we obtain the following result.
\begin{lemma}[Tangential Problem Cauchy Decrease Condition]
The step $p_k$ 
computed by Algorithm~\ref{algorithmBO} satisfies
\be \hpred_k(p_k) \geq \frac12\| (\tg_{k}+\tW_k v_k)^T\tZ_k \| \min\left( (1 - \zeta) \Delta_k,\frac{\| (\tg_{k}+\tW_k v_k)^T\tZ_k \|}{\| 
\tW_k \|}\right).\label{hpred_cauchy}\ee
\end{lemma}


Next, we prove a technical lemma relating the length of the normal step and the predicted feasibility reduction $\vpred_k$.

\begin{lemma}
\label{vk_vs_vpred}
Suppose that Assumptions~2 and 4 hold. 
Then, 
\be \| v_k \| \leq  \Gamma_1 \vpred_k, \label{asm11}\ee
where
\be \Gamma_1 := \frac{2}{\sigma_{\min} \min(1,\kappa^{-2}/2)}\quad\mbox{and}\quad \kappa:= \frac{\sigma_{\max}}{\sigma_{\min}}. \label{ga1def} \ee
\end{lemma}

\begin{proof}
Recalling the Cauchy decrease condition \eqref{lem1_res} we have 
\be \label{vpred_com}
\begin{split}
\vpred_k(p_k) 
		      &\geq \frac{\| \tilde A_k ^T  \tilde c_k \|}{2\| \tilde c_k \|}\min\left(\zeta\Delta_k,\frac{\| \tilde A_k ^T  \tilde c_k \|}{\| \tilde A_k ^T \tilde A_k \|}\right)\\
		      &\geq \frac{\sigma_{\min}}{2} \min\left(\zeta\Delta_k,\frac{\sigma_{\min}\| \tilde c_k \|}{\sigma_{\max}^2 }\right).
\end{split}
\ee
First consider the case where 
\benn \| \tilde c_k \|\geq \tfrac\zeta2 \sigma_{\min} \Delta_k. \eenn
By \eqref{vprobtr} we have 
\be \label{vpred_con1}
\begin{split}
\vpred_k(p_k) 
                       &\geq \frac{\sigma_{\min}}{2} \min\left(\zeta\Delta_k,\frac{\zeta\sigma^2_{\min}}{2\sigma_{\max}^2 } \Delta_k\right)\\
                       &=\frac{\sigma_{\min} \zeta \Delta_k}{2}\min(1,\kappa^{-2}/2) \\
                       &\geq \frac{\sigma_{\min}  }{2}\min(1,\kappa^{-2}/2) \|v_k\|.
\end{split}
\ee
On the other hand, if 
\benn \| \tilde c_k \|< \tfrac\zeta2 \sigma_{\min} \Delta_k \ \Longrightarrow \ 
 \zeta  \Delta_k > \frac{2}{\sigma_{\min}} \| \tilde c_k \|.  \eenn
Substituting in \eqref{vpred_com} we obtain
\be \label{vpred_com2}
\begin{split}
\vpred_k(p_k) &\geq \frac{\sigma_{\min}}{2} \min\left(\zeta\Delta_k,\frac{\sigma_{\min}\| \tilde c_k \|}{\sigma_{\max}^2 }\right)\\
                       &\geq \frac{\sigma_{\min}}{2} \min\left( \frac{2}{\sigma_{\min}} \| \tilde c_k \| ,\frac{\sigma_{\min}\| \tilde c_k \|}{\sigma_{\max}^2 }\right)\\
                       &= \| \tilde c_k \|\min\left(1, \kappa^{-2}/2\right).
\end{split}
\ee
Now, since $v_k$ solves the normal subproblem \eqref{vprob}, 
\benn 
\| \tilde c_k \|^2  \geq \| \tilde c_k + \tilde A_k ^T v_k\|^2 = \| \tilde c_k \|^2 + 2  \tilde c_k ^T  \tilde A_k ^T v_k + \| \tilde A_k ^T v_k\|^2 ,
\eenn
so that
\benn 
-2  \tilde c_k ^T  \tilde A_k ^T v_k \geq\| \tilde A_k ^T v_k\|^2 ,
\eenn
and by the Cauchy-Schwarz inequality we obtain
\be \| \tilde A_k ^Tv_k\| \leq 2\| \tilde c_k \|.\ee
Using this in \eqref{vpred_com2} and obtain
\be\label{vpred_con2}
\begin{split}
\vpred_k(p_k) &\geq \| \tilde c_k \|\min\left(1, \kappa^{-2}/2\right)\\
		      &\geq \frac12\|{ \tilde A_k ^T v_k}\|\min\left(1, \kappa^{-2}/2\right)\\
		      &\geq \frac{\sigma_{\min}}{2}\min\left(1, \kappa^{-2}/2\right) \|v_k\| .
\end{split}
\ee
We conclude the proof by \eqref{vpred_con1} and \eqref{vpred_con2}.
\end{proof}
We can now show that the sequence $\{\nu_k\}$ is bounded.


\begin{lemma}
\label{pred_hred}
{
Let Assumptions 1 through 4 be satisfied.  Then, the sequence $\{\nu_k\}$ is bounded and thus there is an integer $k''$ such that, for all $k \ge k''$, $\nu_k$ takes a constant value $\nu_{k''}$. This constant satisfies
\be
\nu_{k''} \leq   \frac{\tau \Gamma_1}{1-{\pi_1}} \left(M_g + \frac{M_W M_c \Gamma_1}{2}\right) := \bar\nu, \label{bar_nu}\ee
where $\Gamma_1$ is defined in \eqref{ga1def}.}
Moreover,
\be \pred_k \geq \Gamma_2 \hpred_k, \label{gam2}\ee
where
\be \Gamma_2 =\left[ 1+\left(M_g +\frac{M_W M_c \Gamma_1}2 \right) \frac{\Gamma_1}{{\pi_1}  \nu_0}\right]^{-1}. \label{def_gamma2} \ee

\end{lemma}

\begin{proof}
\textit{Part 1.} We  apply \cref{vk_vs_vpred}, and we have that by \cref{pred_vpred_hpred,asm11} and Assumption~2,
\be
\begin{split} 
\pred_k &= \nu_k \vpred_k + \hpred_k - \tilde g_k^T v_k - \frac12 v_k^T \tilde W_k v_k  \\
	     &\geq \nu_k \vpred_k + \hpred_k -\|\tilde g_k\| \|v_k \| - \frac12 \|v_k\|^2 \|\tilde W_k \|  \\
	     &\geq \nu_k \vpred_k + \hpred_k -M_g \|v_k \| - \frac12 \|v_k\|^2 M_W\\
	     &\geq \nu_k \vpred_k + \hpred_k -\left(M_g +\frac{M_W \Gamma_1\vpred_k}2 \right)  \Gamma_1\vpred_k\\
	     &\geq \nu_k \vpred_k + \hpred_k -\left(M_g +\frac{M_W M_c \Gamma_1}2 \right) \Gamma_1\vpred_k, \label{pred_coM_Wong}
 \end{split} 
\ee
where the last inequality follows from the fact that $\vpred_k \leq \|c_k\|$,  by definition \eqref{vpred}. 

Recall that $\nu_k$ is increased until
\be \label{pred_0}
\pred_k \geq {\pi_1}\nu_k \vpred_k .
\ee
By  \eqref{pred_coM_Wong} and the fact that $\hpred_k$ is non-negative, we have that \eqref{pred_0} is satisfied if
\be \label{pred_1}
\nu_k \vpred_k  -\left(M_g +\frac{M_W M_c \Gamma_1}2 \right) \Gamma_1\vpred_k
\geq {\pi_1}\nu_k \vpred_k
\ee
$i.e.$ if
\be \nu_k \geq \frac{ \Gamma_1}{1-{\pi_1}} \left(M_g +\frac{M_W M_c \Gamma_1}2 \right).\ee
Recalling that $\tau$ is the factor by which $\nu_k$ is increased, we conclude that the penalty parameter is never larger than (as defined in \cref{bar_nu}):
\be  
\bar\nu =: \frac{\tau \Gamma_1}{1-{\pi_1}} \left(M_g +\frac{M_W M_c \Gamma_1}2 \right).
\ee
The proof of the first part of the lemma is complete. 

\textit{Part 2.} For the second part of the theorem, we substitute \cref{pred_0} into \cref{pred_coM_Wong}:
\be
\begin{split} 
\pred_k &\geq \nu_k \vpred_k + \hpred_k -\left(M_g +\frac{M_W M_c \Gamma_1}2 \right) \Gamma_1\vpred_k \\
	    & \geq  \hpred_k -\left(M_g +\frac{M_W M_c \Gamma_1}2 \right) \Gamma_1\vpred_k \\
	    & \geq  \hpred_k -\left(M_g +\frac{M_W M_c \Gamma_1}2 \right) \frac{\Gamma_1}{{\pi_1} \nu_k}\pred_k.
\end{split} 
\ee
Re-arranging  
\be
\begin{split} 
\pred_k &\geq \left[ 1+\left(M_g +\frac{M_W M_c \Gamma_1}2 \right) \frac{\Gamma_1}{{\pi_1} \nu_k}\right]^{-1} \hpred_k \\
	     &\geq \left[ 1+\left(M_g +\frac{M_W M_c \Gamma_1}2 \right) \frac{\Gamma_1}{{\pi_1}  \nu_0}\right]^{-1} \hpred_k\\
        &=\Gamma_2 \hpred_k .
\end{split} \nonumber
\ee


\end{proof}



\noindent \textit{Remark 3. The Settling Iterate.} The integer $k''$ after which the penalty parameter is fixed (at a value no greater than $\bar \nu$) will be referred to as the \textit{settling iterate.} We emphasize the distinction between $k'$ and $k''$. The \textit{anchor iterate} $k'$  defined in Remark~1, can be chosen arbitrarily and determines the value of $\nu_{k'}$, which in turn defines the convergence regions. In contrast, $k''$ is significant only in that it exists, so that the merit function is a fixed function for sufficiently large $k$.
\medskip


We next show that when the merit parameter has stabilized and when the reduced gradient is  sufficiently large, the trust region radius cannot be decreased below a certain value. For ease of notation, we define a few quantities.
\be \Theta = \frac{ \pi_0\Gamma_2^2   (1-\zeta)^2 }{2\tau\xi  M_L(\bar\nu)};\quad \mathcal{E}_h = \frac{\xi }{  \Gamma_2 (1-\zeta)} (\eg + \bar\nu \eA);\quad {\bar\varepsilon} = \eps_f + \bar \nu \eps_c\label{defi1_h}.\ee
We also recall from \eqref{def_ML} that $M_L(\nu_k) = \max(L_f + M_W,\ \nu_k L_c)$. 


\begin{lemma}[Increase of the Trust Region in Tangential Problem] 
\label{tr-lb-hprob}
Suppose that for an iterate $k$ and a given positive constant ${\hat\gamma}$,
\be \| (\tg_k +\tW_k v_k )^{T} \tZ_k \| > \mathcal{E}_h  + {\hat\gamma} ,\label{asmp_h_del}\ee
where $\Gamma_2$ is given in \eqref{def_gamma2}. Define
{\be\label{def_hat_del} \hat\Delta ({\hat\gamma}) = \frac{ \Gamma_2   (1-\zeta) }{\xi  M_L(\bar\nu)}{\hat\gamma}.\ee }
Then,
\be \min\left( \hat \Delta ({\hat\gamma}),\frac{\| (\tg_{k}+\tW_k v_k)^T\tZ_k \|}{\| \tW_k \|}\right)  = \hat \Delta ({\hat\gamma}). \label{min_delta_bar_h}\ee
Furthermore, if $\Delta_k \leq \hat\Delta({\hat\gamma})$, the step $p_k$ is accepted and
\be \Delta_{k+1} =\tau \Delta_k. \ee

\end{lemma}

\begin{proof}
 Note from \eqref{def_gamma2} that $\Gamma_2 < 1$.
By \eqref{def_hat_del}  and the definition of $\xi$ in line 3 of Algorithm~1, 
\be
 \hat\Delta ({\hat\gamma}) = \frac{ \Gamma_2   (1-\zeta) }{\xi  M_L(\bar\nu)}{\hat\gamma} 
 \leq \frac{{\hat\gamma}}{M_W}
 < \frac{\| (\tg_{k}+\tW_k v_k)^T\tZ_k \|}{\| \tW_k \|},
 \ee
 where the {first} inequality is obtained by dropping constants that are less than 1 and by definition of $M_L(\nu_k)$. Hence \eqref{min_delta_bar_h} holds.
 
 
Now, since  $\Delta_k \leq \hat\Delta({\hat\gamma})$ and  $1-\zeta < 1$, it follows that
  \be \min\left( (1-\zeta)\Delta_k,\frac{\| (\tg_{k}+\tW_k v_k)^T\tZ_k \|}{\| \tW_k \|}\right)  = (1-\zeta)\Delta_k. \label{smaller_delta_h}\ee
We also have
\be
\begin{split} 
M_L(\nu_k) \Delta_k + (\eg + \bar\nu\eA)
&\leq M_L(\nu_k) \hat\Delta({\hat\gamma}) + (\eg + \bar\nu\eA) \\
&\stackrel{\leq}{\mtiny{\nu_k \leq\bar\nu}} M_L(\bar\nu) \hat \Delta({\hat\gamma}) + (\eg + \bar\nu\eA) \\
&\stackrel{= }{\mtiny{\labelcref{def_hat_del}}} \frac{ \Gamma_2   (1-\zeta) }{\xi }{\hat\gamma} + (\eg + \bar\nu\eA) \\
&=  \frac{ \Gamma_2   (1-\zeta) }{\xi }\left[{\hat\gamma} + \frac{\xi}{\Gamma_2 (1-\zeta) }(\eg + \bar\nu\eA) \right]\\
&=  \frac{ \Gamma_2   (1-\zeta) }{\xi }\left[{\hat\gamma} + \mathcal{E}_h \right]\\
&< \frac{ \Gamma_2   (1-\zeta) }{\xi } \| (\tg_k +\tW_k v_k )^{T} \tZ_k \|.
\end{split}
\label{lem3_aux_h}
\ee
Using this bound and the definition of $\rho$, 
we obtain
\be
\begin{split}
|\rho_k-1| &= \frac{|\pred_k(p_k)-\ared_k(p_k)|}{|\pred_k(p_k) + \xi(\ef+\nu_k  \ec) |}\\ 
& \stackrel{\leq}{\mtiny{ \labelcref{gam2}}} \frac{|\pred_k(p_k)-\ared_k(p_k)|}{\Gamma_2  \hpred_k(p_k) +\xi(\ef+\nu_k   \ec) } \\
& \stackrel{\leq}{ \mtiny{ \labelcref{hpred_cauchy},\labelcref{lem2_res}}} \frac{M_L(\nu_k) \Delta_k^2 + (\eg + \nu_k \eA) \Delta_k + 2(\ef + \nu_k \ec)}{  \frac{\Gamma_2 }2 \| (\tg_{k}+\tW_k v_k)^T\tZ_k \| \min\left( (1 - \zeta) \Delta_k,\frac{\| (\tg_{k}+\tW_k v_k)^T\tZ_k \|}{\| \tW_k \|}\right)+ \xi(\ef+\nu_k  \ec)} \\
& \stackrel{\leq}{\mtiny{\nu_k \leq\bar\nu}} \frac{M_L(\nu_k) \Delta_k^2 + (\eg + \bar\nu \eA) \Delta_k + 2(\ef + \nu_k \ec)}{  \frac{\Gamma_2 }2 \| (\tg_{k}+\tW_k v_k)^T\tZ_k \| \min\left( (1 - \zeta) \Delta_k,\frac{\| (\tg_{k}+\tW_k v_k)^T\tZ_k \|}{\| \tW_k \|}\right)+ \xi(\ef+\nu_k  \ec)} \\
& \stackrel{=}{  \mtiny{ \labelcref{smaller_delta_h} }   } \frac{[M_L(\nu_k) \Delta_k + (\eg + \bar\nu \eA)] \Delta_k + 2(\ef + \nu_k \ec)}{  \frac{\Gamma_2 (1 - \zeta)}2 \| (\tg_{k}+\tW_k v_k)^T\tZ_k \|   \Delta_k + \xi(\ef+\nu_k  \ec)}\\
& \stackrel{<}{\mtiny{ \labelcref{lem3_aux_h} }} \frac{\frac{ \Gamma_2   (1-\zeta) }{\xi } \| (\tg_k +\tW_k v_k )^{T} \tZ_k \| \Delta_k + 2(\ef + \nu_k \ec)}{  \frac{\Gamma_2 (1 - \zeta)}2 \| (\tg_{k}+\tW_k v_k)^T\tZ_k \|   \Delta_k + \xi(\ef+\nu_k  \ec)}\\
& = \frac2\xi\\
& = 1-\pi_0.
\end{split}\ee
By line 16 of Algorithm~1, the step is accepted.

\end{proof}

\begin{corollary}[Lower Bound of Trust Region Radius]\label{tr_lb_h}
Given ${\hat\gamma}>0$, if there exist $ K>0$ such that for all $k\geq K$
\be \| (\tg_k +\tW_k v_k )^{T} \tZ_k \| > \mathcal{E}_h  + {\hat\gamma} ,\ee
then there exist $\hat K \geq K $ such that for all $k\geq \hat K$,
\be \Delta_k > \tfrac{1}{\tau}\hat\Delta ({\hat\gamma}).\label{tr-bardel-h} \ee
\end{corollary}

\begin{proof}
We apply \cref{tr-lb-hprob} for each iterate after $K$ to deduce that, whenever $\Delta_k \leq \hat \Delta({\hat\gamma})$, the trust region radius will be increased. Thus, there is an index $\hat K$ for which $\Delta_k$ becomes greater than $\hat\Delta({\hat\gamma})$. On subsequent iterations, the trust region radius can never be reduced below $\hat\Delta({\hat\gamma})/\tau$ (by Step~6 of Algorithm~\ref{algorithmBO}) establishing \eqref{tr-bardel-h}. 
\end{proof}

Additionally, for any given $\mu >0$, define
 \be \label{defi2_h} {\bar\gamma} = \frac12 \left(-\mE_{h} + \sqrt{\mE_{h}^2 + 8 (\eps_f + \bar\nu \eps_c) / \Theta} \right)+ \mu.
\ee

\begin{lemma}[Merit Function Reduction in Tangential Problem]\label{noisyfuncred_h}
{Let Assumptions 1 through 4 be satisfied. Let $ k' $ denote the anchor iterate and $k''$ the settling iterate, as defined above. Suppose for some $k\geq\max(k',k'')$,}
\be\label{supp_h}\| (\tg_k +\tW_k v_k )^{T} \tZ_k \| > \mathcal{E}_h+ {\bar\gamma},\quad \text{and}\quad \Delta_k \, {\geq} \, \frac{\hat\Delta({{\bar\gamma})}}{\tau} , \ee 
 where $\hat \Delta (\cdot)$ is defined in \eqref{def_hat_del} and ${\bar\gamma}$ is defined in \eqref{defi2_h} 
 with $\mu >0$  an arbitrary  constant.
Then, 
%
\be \label{hpred_gamma}
\hpred_k(p_k) 
\geq \frac{ \Theta}{\pi_0\Gamma_2}\left(\mathcal{E}_h  + {\bar\gamma}\right){\bar\gamma}.
\ee
Furthermore, if the step is accepted at iteration $k$, we have
\be
\tilde \phi(x_k , \nu_k) - \tilde \phi(x_k + p_k , \nu_k) > \Theta \mu^2 +\mu \sqrt{\Theta^2\mE_{h}^2 + 8\Theta(\eps_f + \nu_{k''} \eps_c)}.
\ee

\end{lemma}

\begin{proof} 
Since inequality \eqref{asmp_h_del} is satisfied, so is \eqref{min_delta_bar_h}. Thus
\be\label{tr-lem36_1_h}
    \begin{split}
        \min \left((1-\zeta)\Delta_k,\frac{\|(\tg_k+\tW_k v_k)^T \tZ_k\|}{\|\tW_k\|}\right)
        &\stackrel{\geq}{\mtiny{ \labelcref{supp_h} }} \min\left(\frac{1-\zeta}{\tau} \hat\Delta({\bar\gamma}),\frac{\|(\tg_k+\tW_k v_k)^T \tZ_k\|}{\|\tW_k\|}\right)\\
        &\stackrel{\geq}{\mtiny{ \labelcref{min_delta_bar_h} }} \frac{1-\zeta}{\tau} \hat\Delta({\bar\gamma}) \\
        &\stackrel{=}{\mtiny{ \labelcref{def_hat_del} }}\frac{ \Gamma_2   (1-\zeta)^2 }{\tau\xi  M_L(\bar\nu)}{\bar\gamma}.
    \end{split}
\ee
By \eqref{hpred_cauchy}, 
\benn 
\begin{split}
\hpred_k(p_k) 
    &\geq   \frac{1}2 \| (\tg_{k}+\tW_k v_k)^T\tZ_k \| \min\left( (1 - \zeta) \Delta_k,\frac{\| (\tg_{k}+\tW_k v_k)^T\tZ_k \|}{\| \tW_k \|}\right)\\
    & \stackrel{\geq}{\mtiny{\labelcref{supp_h}\labelcref{tr-lem36_1_h}}}
         \frac{1}{2} \left(\mathcal{E}_h  + {\bar\gamma}\right)\frac{ \Gamma_2   (1-\zeta)^2 }{\tau\xi  M_L(\bar\nu)}{\bar\gamma}\\
    & = 
    	\frac{ \Gamma_2   (1-\zeta)^2 }{2\tau\xi  M_L(\bar\nu)}\left(\mathcal{E}_h  + {\bar\gamma}\right){\bar\gamma}\\
     &= 
        \frac{ \Theta}{\pi_0\Gamma_2}\left(\mathcal{E}_h  + {\bar\gamma}\right){\bar\gamma},
\end{split}
\eenn
which proves the first part of the lemma.

Now, by Assumption 4 the singular values of $\tA_k$ are bounded below by $\sigma_{\min}$ and above by $\sigma_{\max}$. Therefore \cref{vk_vs_vpred,pred_hred} apply, and by \eqref{gam2}
we have that $ \pred_k \geq \Gamma_2 \hpred_k.$
Let a step be accepted. Then, as explained in \eqref{redex}, 
\be 
    \ared_k >  \pi_0 \pred_k -2(\eps_f + \nu_k \eps_c) = \pi_0 \pred_k -2(\eps_f + \nu_{k''} \eps_c) ,
\ee
and thus
\be 
\tilde \phi(x_k , \nu_k) - \tilde \phi(x_k + p_k , \nu_k) > \pi_0 \Gamma_2 \hpred_k - 2(\eps_f + \nu_{k''} \eps_c).
\ee
Using condition \eqref{hpred_gamma} 
we obtain 
\be
\begin{split}
    & \tilde \phi(x_k , \nu_k) - \tilde \phi(x_k + p_k , \nu_k) \\
    & > \pi_0 \Gamma_2 \hpred_k - 2(\eps_f + \nu_{k''} \eps_c)\\
    &= \Theta(\mathcal{E}_h+{\bar\gamma}){\bar\gamma} - 2(\eps_f + \nu_{k''} \eps_c)\\
    &= \Theta \mu^2 +\mu \sqrt{\Theta^2\mE_{h}^2 + 8\Theta (\eps_f + \nu_{k''} \eps_c)},
    \end{split}
\ee
    where the last equality follows as in the derivation of \eqref{leftout}.

\end{proof}

We now study the achievable reduction in the norm of the reduced gradient, $Z(x)^T g(x)$. Recall that $Z(x)$ and $\tilde Z(x)$ denote orthonormal bases for the null spaces of $A(x)$ and $\tA(x)$, respectively. We define
\be \tilde Z(x) - Z(x) = \delta_Z(x),
\ee
and make the following assumption.

\medskip\noindent \textbf{Assumption 5.} There exist constant $\eps_Z$ such that:
\be \|\delta_Z(x)\| \leq \eps_Z.\ee

One can  satisfy this assumption in practice if the same pivoting order is used in the QR factorization that computes $Z$. Or when $Z$ is not required to be orthonormal, we can achieve this by using the same basic/nonbasic set, as discussed in the appendix.

We add some additional comments about this assumption. This assumption is realistic and can be satisfied by specific choices of computing $Z(x)$ from $A(x)$ as long as the quantities being computed are well defined. Furthermore, the new quantity $\eps_Z$ in this assumption can be shown to depend on, for instance, $\eps_A$ and conditioning numbers of the matrices. 

To bound the differences between $\tilde Z(x)^T \tilde g(x)$ and $Z(x)^Tg(x)$, we also define 
\be \bar G^{II}_{Ac}(k') = \sup_{x\in C_{Ac}^{II}(k')} {\|g(x)\|}.\ee

\begin{lemma}\label{gZ_diff_bd}
    Let Assumptions 1 through 5 be satisfied. If $x\in C_{Ac}^{II}(k')$, then        
        \be \|g(x)^T Z(x) -\tg(x)^T \tZ(x)\| \leq \varepsilon_{gZ}, \quad\text{where} \quad \varepsilon_{gZ} = \eps_g  + \eps_Z \bar G^{II}_{Ac}(k') + \eps_g\eps_Z. \ee
\end{lemma}
\begin{proof} We have that 
    \be \begin{split}
            \|g(x)^T Z(x) -\tg(x)^T \tZ(x)\|  &= \|g(x)^TZ(x) - [g(x)+\delta_g(x)]^T[Z(x)+\delta_Z(x)]\| \\
                                          &= \| -\delta_g(x)^T Z(x) + g(x)^T \delta_Z(x)  + \delta_g(x)^T\delta_Z(x) \| \\
                                          &\leq \eps_g  + \eps_Z \| g(x) \|+ \eps_g\eps_Z  \\
                                          &\leq \eps_g  + \eps_Z \bar G^{II}_{Ac}(k') + \eps_g\eps_Z \\
                                          &= \varepsilon_{gZ}.
        \end{split}\ee
\end{proof}

\begin{proposition}[Finite Time Entry to Critical Region 1 of Optimality]
\label{prop3}
Suppose that Assumptions 1 through 5 hold. Let $ k' $ denote the anchor iterate and $k''$ the settling iterate. Then, once the sequence  $\{x_k\}$ generated by Algorithm 1 visits $C_{Ac}^I(k')$ for the first time, it visits infinitely often the region $C^I_{gZ}$ defined as
\be C^I_{gZ}:= \{ x| \|g(x)^{T} Z(x) \|\leq \mathcal{E}_{gZ}^{I}\}, \ee
where 
\be \mathcal{E}_{gZ}^{I}(k',k'')= \mathcal{E}_h   +\frac{\Gamma_1 L_W\mathcal{E}_{Ac}^{II}}{\sigma_{\min}}+\varepsilon_{gZ}+ {\bar\gamma}, \label{EgZI}\ee
and, as before,
\be \mathcal{E}_{Ac}^{II} = \sup_{x\in C_{Ac}^{II}} \|\tA(x)^T \tc(x)\|.\label{cEAcII}\ee
\end{proposition}

\begin{proof} Recall that once the iterates enter $C_{Ac}^{I}$, by \cref{prop2}, they remain in $C_{Ac}^{II}$. Thus, since the singular values of $\tA_k$ are assumed to bounded below by $\sigma_{\min}>0$,  we have from the definition \eqref{cEAcII}
\be \|c_k\| \leq \frac{\mathcal{E}_{Ac}^{II}}{\sigma_{\min}}. \ee
Applying condition \eqref{asm11} from \cref{vk_vs_vpred} we obtain
\be\| v_k \| \leq \Gamma_1 \vpred_k\leq \Gamma_1 \|c_k\| \leq \frac{\Gamma_1\mathcal{E}_{Ac}^{II}}{\sigma_{\min}}.\ee 
Therefore, 
\be\| (\tW_k v_k )^{T} \tZ_k \| \leq L_W \|v_k\|\|\tZ_k\| \leq \frac{\Gamma_1 L_W\mathcal{E}_{Ac}^{II}}{\sigma_{\min}}.\label{Wv_bd}\ee

We now proceed by  means of contradiction and assume that there exist an integer $K$, such that for all $k>K$, none of the iterates is in $C_{gZ}^{I}$, i.e.,
\be \|g_k^T Z_k\| >  \mathcal{E}_{gZ}^{I}, \ee
and by \cref{gZ_diff_bd},
\be \|\tg_k^T \tZ_k\| >  \mathcal{E}_{gZ}^{I}- \varepsilon_{gZ}. \label{c1h_asp}\ee 
Thus, for all $k>K$:
\be \label{gwvZ}
\begin{split}
    \| (\tg_k +\tW_k v_k )^{T} \tZ_k \| 
        & \geq \| \tg_k^T \tZ_k\| -\| (\tW_k v_k )^{T} \tZ_k  \|\\
        &\stackrel{\geq}{\mtiny{ \labelcref{c1h_asp},\labelcref{Wv_bd} }}\mathcal{E}_{gZ}^{I} - \varepsilon_{gZ}- \frac{\Gamma_1 L_W\mathcal{E}_{Ac}^{II}}{\sigma_{\min}}\\
        &\stackrel{=}{\mtiny{\labelcref{EgZI}}}  \mathcal{E}_h   + {\bar\gamma}.
\end{split}
\ee
Therefore, \cref{tr_lb_h} applies showing the existence of a lower bound for the trust region radii 
for a sufficiently large $k$. This implies that there will be infinitely many accepted steps (for otherwise $\Delta_k\rightarrow 0$).
For $k$ large enough and for each of the accepted steps we have by \cref{noisyfuncred_h} that
\be
\tilde \phi(x_k , \nu_k) - \tilde \phi(x_k + p_k , \nu_k) > \Theta \mu^2 +\mu \sqrt{\Theta^2\mE_{h}^2 + 8\Theta (\eps_f + \nu_{k''} \eps_c)}.
\ee
Since this inequality holds infinitely often,  $\{\tilde \phi(x_k, \nu_k)\}$ is unbounded below, which is a contradiction. Therefore our assumption is incorrect, proving the iterates visits $C_{gZ}^I$ infinitely often.

\end{proof}

\begin{lemma}[Bound on Displacement Outside of Critical Region I of Optimality]\label{disp_h}
Let Assumptions 1 through 5 be satisfied and let $k',k''$ be the anchor and settling iterates, respectively. Let $k_1 \geq \max{(k',k'')}$ be such that $x_{k_1} \in C_{gZ}^{I}$  and $x_{k_{1}+1}\notin C_{gZ}^{I}$.
Then, if $\Delta_{k_1} < \hat\Delta_{{\bar\gamma}}$,  there exist a finite iterate $k_2\geq k_{1}+1$, defined as
\be \label{k2h}
k_2 = \min \left\{ k \geq k_1 + 1 : \Delta_k \geq \hat{\Delta}_{k''} \text{ or } x_k \in C_{gZ}^I \right\}.
\ee
Furthermore, for any $k$ with $k_1 \leq k \leq k_2$ we have that
\be \| x_{k} - x_{k_1} \| \leq \frac{\tau}{\tau-1} \hat\Delta_{k''} \ee
\end{lemma}

\begin{proof}

We  show the first part of the lemma by  means of contradiction. Assume for contradiction that $k_2$ is not finite. Therefore, for $k = k_1+1, k_1+2 , ... $, 
\be \Delta_k < \hat\Delta_{k''}\label{lm5_delta_bd_h}\ee 
and
\be x_k\notin C_{gZ}^{I}(k''),\ee 
which as argued in \eqref{gwvZ}, implies
\be\| (\tg_k +\tW_k v_k )^{T} \tZ_k \| \geq \mathcal{E}_h   + {\bar\gamma}.
\label{tswim}\ee
Therefore we apply \cref{tr-lb-hprob} for each iterate {$k \geq k_1+1$} and obtain that $\Delta_k\rightarrow \infty$ as $k\rightarrow \infty$, contradicting \eqref{lm5_delta_bd_h}.

For the rest of the lemma, we take any $k$ with $k_1 < k < k_2$ and have that $x_k \notin C_{gZ}^{I}$,
thus \eqref{tswim} holds.
Also, by assumption for each of these iterates $k$,
$\Delta_{k} < \hat\Delta_{k''}.$ Therefore by \cref{tr-lb-hprob}, 
$ \Delta_{k+1} = \tau \Delta_k .$
Also note $\Delta_{k_2-1} < \hat\Delta_{k''}$, thus for $i = 0, 1, ..., k_2-k_1-1$
\be \Delta_{k_2-1-i} = \tau^{-i} \Delta_{k_2-1} < \tau^{-i} \hat\Delta_{k''}.\ee
Therefore
\be
\begin{split}
    \|x_{k} - x_{k_1} \| &\leq \sum_{i = 1}^{k - k_1} \|x_{k_1 + i} - x_{k_1 + i-1}\|\\
				&\leq \sum_{i = 1}^{k_2 - k_1} \|x_{k_1 + i} - x_{k_1 + i-1}\|\\
				&\leq \sum_{j = k_1}^{k_2-1} \Delta_{j}\\
			        & = \sum_{i = 0}^{ k_2 - 1 - k_1} \tau^{-i} \Delta_{k_2-1}\\
    				& < \hat\Delta_{k''}\sum_{i = 0}^{\infty}\tau^{-i}\\
    				& = \frac{\tau}{\tau-1} \hat\Delta_{k''}.
\end{split}
\ee
\end{proof}

We now define a maximum value of the re-scaled merit function $\tilde\phi(x,\nu_{k''})$ in $C_{gZ}^{I}(k')$. In particular,
\be\label{phiIgZ} \bar\phi_{gZ}^I = \sup_{x\in C_{gZ}^{I}} \phi(x , \nu_{k''}) \ee
Furthermore, we define a maximum value of the gradient of the objective function in $C_{gZ}^{I}(k')$ as
\be \bar G^I_{gZ} = \sup_{x\in C_{gZ}^{I}} \|g(x)\|.\ee
The next proposition we present will demonstrate that the iterates cannot stray too far from stationary points in the sense that the merit function is bounded. For this bound, we shall state the result for the merit function problem without noise:
\be \label{scale}\phi(x) = f(x) + \nu c(x).\ee





\begin{proposition}[Remaining in Critical Region II of Feasibility]\label{prop4}
Once an iterate enters  $C_{gZ}^{I}$, it never leaves the set $C_{gZ}^{II}$ defined as
\be
 C_{gZ}^{II} =\left\{x: \phi( x , \nu )  \leq  \bar\phi_{gZ}^I +  \max(\mathcal{P}_{gZ}^{II} , 2 {\bar\varepsilon}  )+ 2 {\bar\varepsilon}  := E_{gZ}^{II}\right\},
\label{EgZII}
\ee
where {$\phi$ is defined in \cref{scale}} and
\be \mathcal{P}_{gZ}^{II}=\left[    \bar G^I_{gZ} + \nu_{k''} \mE_{Ac}^{II} +      \frac{ \tau\Gamma_2   (1-\zeta) }{(1-\tau)\xi  M_L(\bar\nu)}{\bar\gamma}  \right]  \frac{ \tau\Gamma_2   (1-\zeta) }{(1-\tau)\xi  M_L(\bar\nu)}{\bar\gamma}.\ee
\end{proposition}

\begin{proof}

We let $k_1$ and $k_2$  be defined as in the last lemma:
\be x_{k_1}\in C_{gZ}^{I}(k'), \quad x_{k_1+1}\notin C_{gZ}^{I}(k'),\ee
\be \label{k2_prop2_h}
k_2 = \min \left\{ k \geq k_1 + 1 : \Delta_k \geq \hat{\Delta}_{k''} \text{ or } x_k \in C_{gZ}^I \right\},
\ee
and recall that $k_2$ is finite.

Since we consider only iterates $k$ with  $k\geq k''$, at which point the merit parameter has attained its final value $\nu_k = \nu_{k''} \leq \bar \nu$, we have for $k = k_1, .... $
\be \label{phi_dist_h}
\begin{split}
    |\tilde \phi(x_{k},\nu_k) - \phi(x_{k},\nu_k) |
      \leq& | \delta_f(x_{k})| + \nu_k\| \delta_c(x_{k}) \|\\
    \leq& \eps_f + \nu_{k}  \eps_c\\
    \leq & \eps_f + \bar\nu \eps_c\\
    = & {\bar\varepsilon}.
\end{split}
\ee
Since the step from $k_1 $ is accepted, we have that \cref{redex}-\cref{redex2} hold for $k= k_1$ and thus
\be \tilde\phi(x_{k_1},\nu_{k_1}) - \tilde\phi(x_{{k_1+1}},\nu_{{k_1+1}}) >  -2(\eps_f + \nu_k\eps_c) \geq -2(\eps_f + \bar\nu\eps_c) = -2{\bar\varepsilon}. \ee

Recalling definition \cref{phiIgZ} and the fact that the $k_1$ iterate is in $C_{gZ}^I(k')$, we have
\be 
\tilde\phi(x_{{k_1+1}},\nu_{{k_1+1}}) 
<  \tilde\phi(x_{k_1},\nu_{k_1}) + 2{\bar\varepsilon}
\stackrel{<}{\mtiny{\labelcref{phi_dist}}} \phi(x_{k_1},\nu_{k_1}) + 3{\bar\varepsilon}
\leq \bar\phi_{gZ}^I(k')+ 3{\bar\varepsilon}. \label{exit_bd_h}\ee

We divide the rest of the proof into two cases based
on whether  $\Delta_{k_1} \geq \hat\Delta_{k''}$ or not.

\medskip
{\bf Assume} $\Delta_{k_1} \geq {\hat{\Delta}_{k''}}$. For each $k = k_1+1, \ldots, k_2-1$, it follows that $x_k \notin C_{gZ}^I(k')$. According to \cref{gwvZ}, this implies $\| (\tg_k +\tW_k v_k )^{T} \tZ_k \| 
        \geq  \mathcal{E}_h + {\bar\gamma}$, so that condition \cref{asmp_h_del} in \cref{tr-lbBO} is satisfied with ${\hat\gamma} = {\bar\gamma}$. Now, for $k \in \{ k_1+1, \ldots, k_2-1\}$ the trust region radius can decrease, but by 
 \cref{tr-lb-hprob}, if at some point $\Delta_k < \hat\Delta_{k''}$ then $\Delta_{k+1} = \tau \Delta_k$. We deduce that $\Delta_k > \frac{\hat{\Delta}_{k''}}{\tau}$ for all  $k \in \{ k_1+1, \ldots, k_2-1\}$.
 We then apply \cref{noisyfuncred_h} 
 to conclude that each accepted step reduces the merit function from $\tilde\phi(x_{{k_1+1}},\nu_{{k_1+1}})$. We have that for each step $k$ after the exiting iterate $k_1+1$,
\be\label{prop3_key1}
\phi(x_{{k}},\nu_{{k}}) \leq \tilde\phi(x_{{k}},\nu_{{k}}) + {\bar\varepsilon} <\tilde\phi(x_{{k_1+1}},\nu_{{k_1+1}})+{\bar\varepsilon} \stackrel{<}{\mtiny{\labelcref{exit_bd_h}}} \bar\phi_{Ac}^I(k')+ 4{\bar\varepsilon} \leq E_{gZ}^{II}. 
\ee
This concludes the proof for when $\Delta_{k_1} \geq {\hat{\Delta}_{k''}}$.

{

\medskip {\bf Assume} $\Delta_{k_1} < \hat\Delta_{k''}$. Using \cref{disp_h}, we can bound the displacement of iterates from $k_1$ to any $k = k_1 +1, \ldots, k_2$. Specifically, by \cref{disp_h}, for $k_1 \leq k \leq k_2$
\be \| x_{k} - x_{k_1} \| \leq \frac{\tau}{\tau-1} \hat\Delta_{k''}. \ee
By $L_f$ and $L_c$--Lipschitz differentiability of the objective and the constraints, respectively, we have for any $k= k_1 , ... , k_2$:
\be 
\begin{split}
f(x_k) - f(x_{k_1}) & \leq \max_{t\in[0,1]}\|g(t x_{k_1} +(1-t) x_k ) \| \|x_k - x_{k_1}\|\\
		      & \leq \left[\|g(x_{k_1})\| + L_f  \| \|x_k - x_{k_1}\|\right] \| \|x_k - x_{k_1}\|\\
		      & \leq \left[\bar G^I_{gZ} + \frac{\tau L_f}{\tau-1}\hat\Delta_{k''}\right] \frac{\tau}{\tau-1}\hat\Delta_{k''}.
\end{split}
\ee
Similarly, for any $k_1 \leq k \leq k_2$, 
\be 
\begin{split}
\|c(x_k)\| - \| c(x_{k_1})\| & \leq \max_{t\in[0,1]}\|\grad c(t x_{k_1} +(1-t) x_k ) \| \|x_k - x_{k_1}\|\\
		      & \leq \left[\| A^T(x_{k_1})c(x_{k_1})\| + L_c  \| \|x_k - x_{k_1}\|\right] \| \|x_k - x_{k_1}\|\\
		      & \leq \left[\mE_{Ac}^{II}  + \frac{\tau L_c}{\tau-1}\hat\Delta_{k''}\right] \frac{\tau}{\tau-1}\hat\Delta_{k''}.
\end{split}
\ee}
Using these two last results and recalling the definition \cref{defi2_h} of $\hat \Delta_{k''}$, and that $k_1\geq k''$ so that the merit parameter settles at $\nu_{k''}$, we find, for any $k_1 \leq k \leq k_2$,

\be \label{prop3_res1}
\begin{split}
&\phi(x_k,\nu_{k}) - \phi(x_{k_1},\nu_{k_1})\\
	& = [ f(x_k) -  f(x_{k_1}) ]+\nu_{k''} [\|c(x_k)\| - \| c(x_{k_1})\|] \\
	& \leq  \left[\bar G^I_{gZ} + \frac{\tau L_f}{\tau-1}\hat\Delta_{k''}\right] \frac{\tau}{\tau-1}\hat\Delta_{k''} + \nu_{k''} \left[\mE_{Ac}^{II} + \frac{\tau L_c}{\tau-1}\hat\Delta_{k''}\right] \frac{\tau}{\tau-1}\hat\Delta_{k''}\\
	& =   \left[   \bar G^I_{gZ}   + \nu_{k''}\mE_{Ac}^{II} +     \left(  L_f + \nu_{k''} L_c\right) \frac{\tau \hat\Delta_{k''}}{\tau-1}  \right] \frac{\tau \hat\Delta_{k''}}{\tau-1}\\
	& =  \left[    \bar G^I_{gZ} + \nu_{k''} \mE_{Ac}^{II} +      \frac{ \tau\Gamma_2   (1-\zeta) }{(1-\tau)\xi  M_L(\bar\nu)}{\bar\gamma}  \right]  \frac{ \tau\Gamma_2   (1-\zeta) }{(1-\tau)\xi  M_L(\bar\nu)}{\bar\gamma}\\
	&= \mathcal{P}_{gZ}^{II}
\end{split}
\ee

Therefore we find for any $k_1 \leq k \leq k_2$,
\be\label{prop3_key2}
\begin{split}
    \phi(x_k,\nu_k) &\leq \phi(x_{k_1},\nu_{k_1}) + \mathcal{P}_{gZ}^{II} \\
                    & \stackrel{\leq}{\mtiny{\labelcref{phi_dist}}}  \tilde\phi(x_{k_1},\nu_{k_1}) + \mathcal{P}_{gZ}^{II} + {\bar\varepsilon} \\
                    & \stackrel{\leq}{\mtiny{\labelcref{exit_bd}}} \bar\phi_{gZ}^I(k')+ \mathcal{P}_{gZ}^{II}+ 4{\bar\varepsilon},\\
                    &\leq E_{gZ}^{II}(k').
\end{split}
\ee
If $x_{k_2}\in C_{gZ}^{I}$, the proof is complete. 

On the other hand, if  $x_{k_2}\notin C_{gZ}^{I}$, we only need to show that condition \cref{prop1_key3} is also satisfied by $k=k_2+1, ... ,\hat K$, where 
\be\hat K = \min\{k\geq k_2+1: x_{k}\in C_{gZ}^{I}  \}.\ee
The existence of $\hat K$ is guaranteed by \cref{prop3}.

For this, we first note in particular, let $k = k_2$ in \cref{prop3_res1}:
\be
\phi(x_{k_2},\nu_{k_2})  \leq \phi(x_{k_1},\nu_{k_1})+  \mathcal{P}_{gZ}^{II};
\ee
with \cref{phi_dist_h} this gives
\be\label{phi_k2_h}
\tilde \phi(x_{k_2},\nu_{k_2})  \leq \phi(x_{k_1},\nu_{k_1})+  \mathcal{P}_{gZ}^{II} + {\bar\varepsilon}\leq \bar\phi_{gZ}^I(k')+  \mathcal{P}_{gZ}^{II} + {\bar\varepsilon},
\ee
where the last inequality is due to the fact that $k_1\in C_{gZ}^{I}$.
Since that iterates have not yet returned into $C_{gZ}^{I}$ at iterate $k_2$, we apply \cref{noisyfuncredBO} for each of the iterates after $k_2$ until iterates return to $C_{gZ}^{I}$ again at iterate $\hat K$ (such iterate exist due to \eqref{prop3}) and obtain that 
\be \tilde\phi(x_{k_2},\nu_{k_2}) > \tilde\phi(x_{k_2+1},\nu_{k_2+1}) > ... > \tilde\phi(x_{\hat K},\nu_{\hat K}). \label{prop3_res2}\ee
Recalling again \eqref{phi_dist_h}, we find that for $k = k_2+1, ... , \hat K$,
\be\label{prop3_key3}
\begin{split}
\phi(x_{k},\nu_{k}) &\leq \tilde\phi(x_{k},\nu_{k}) + e_{k'} \\
                    &\stackrel{<}{\mtiny{\labelcref{prop1_res2}}}  \tilde\phi(x_{k_2},\nu_{k_2}) + {\bar\varepsilon}\\
                    &\stackrel{\leq}{\mtiny{\labelcref{phi_k2}}} \bar\Phi_{Ac}^I(k')+  \mathcal{P}_{Ac}^{II} + 2{\bar\varepsilon}
\end{split}
\ee
We now combine results from \cref{prop3_key1,prop3_key2,prop3_key3} and conclude the proof.
\end{proof}

\subsection{Summary of the Convergence Results}
We now recapitulate the results established in this paper.

\begin{theorem}[Final Result] 
Let $\{x_k\}$ be the sequence generated by Algorithm~1.
If Assumptions~1 through 3 are satisfied, the following two results hold:

(i) [Proposition 1] The sequence $\{x_k\}$ visits infinitely often a critical region $C_{Ac}^I$ where the stationary measure of feasibility is small up to noise level: 
\be \|A(x)^T c(x)\|\leq \mE_{Ac}^I. \ee

(ii) [Proposition 2] Once an iterate enters $C_{Ac}^I$, the rest of the iterates remains in a larger (up to scaling of noise level) critical region $C_{Ac}^{II}$, where \be \Phi(x,\nu)\leq E_{Ac}^{II}. \ee

\noindent
If, in addition, Assumptions 4 and 5 hold, then the sequence of merit parameters $\{\nu_k\}$ remains bounded, and we have:

(iii) [Proposition 3] Once the sequence $\{x_k\}$ enters $C_{Ac}^{II}$, it visits infinitely often a critical region $C_{gZ}^I$, where the projected gradient is small up to noise level:
\be \|g(x)^T Z(x)\|\leq \mE_{gZ}^I+\varepsilon_{gZ}; \ee

(iv) [Proposition 4] After the iterates enter $C_{gZ}^I$ for the first time, they  remain in a larger (up to scaling of noise level) critical region $C_{gZ}^{II}$, where \be \phi(x,\nu)\leq E_{gZ}^{II}; \ee
\end{theorem}

\medskip\noindent
We summarize these results in table~\ref{table:2}.


\begin{table}[ht]
\centering
\begin{tabular}{|c|c|c|}
\hline
 & \makecell{\textbf{Critical Region I} \\ \textit{(stationary measure bounded)}} & \makecell{\textbf{Critical Region II} \\ \textit{(merit function bounded)}} \\ \hline
\makecell{ {\textbf{Feasibility ($A^T c$)}}\\ {(under any conditions)} } & $C^I_{Ac}= \{x | \|A(x)^T c(x)\|\leq\mathcal{E}_{Ac}^I$\} & $C^{II}_{Ac}= \{x |\Phi(x,\nu) \leq E_{Ac}^{II}\}$ \\ \hline
\makecell{ {\textbf{Optimality ($g^T Z$)}}\\ {(when $\tA$ is full rank)} }  & $C^I_{gZ}= \{x |\|g(x)^T Z(x)\|\leq\mathcal{E}_{gZ}^I\}$ & $C^{II}_{gZ}= \{x |\Phi(x,\nu) \leq E_{gZ}^{II}\}$ \\ \hline
\end{tabular}
\caption{Convergence Regions }
\label{table:2}
\end{table}


\noindent \textbf{Remark 3: Role of the Merit Parameter.} 
In the absence of the constraints,  $\nu \equiv 0$, and in this  case, $C_{gZ}^I$ and $C_{gZ}^{II}$ reduce to the regions $C_1$ and $C_2$ in \cite{sun2023stochastic}. In the absence of an objective, by sending the merit parameter to arbitrary large value, the rescaled merit function as used in $C_{Ac}^{II}$ becomes arbitrarily close to $\|c(x)\|$, again recovering a result that is expected of the nonlinear equations only investigation. 


\section{Numerical Results}  \label{numerical}

We tested the robustness of the proposed algorithm in the noisy setting. To this end, we employed  {\sc knitro} \cite{ByrdNoceWalt06}, which contains a careful implementation of the BO algorithm that is accessible by setting
{\tt options.algorithm = 2} ({\sc knitro-cg}). The original BO algorithm in {\sc knitro} was modified by Figen Oztoprak from Artelys Corp. to include, as an option, the modified ratio  \eqref{rhodef} and the ability to input the noise level. The default stopping criteria of {\sc knitro} were used, ensuring consistency across all tests.

We tested problems from the standard CUTEst library \cite{gould2015cutest}, 
accessed via the Python interface. To simulate the noisy settings, we inject randomly generated noise in the objective function, the gradient, Hessian and Jacobian. {For each iterate $x_k$, we sample $\delta_f, \delta_c, \delta_g, \delta_J$ from the uniform distribution  $\mathcal{D}(\eps,m,n)$ with a fixed value $\epsilon$ for the noise in $f, c, g, A$, respectively.
Here  $\mathcal{D}(\epsilon, m, n)$ represents an $m \times n$ matrix-valued distribution, where each element is independently drawn from a one-dimensional distribution $\mathcal{D}$ with support in  $[-\epsilon, \epsilon]$. We also tested noise generated by a Gaussian distribution, with the standard deviation in place of the error bounds, with similar results.}


While we conducted experiments on over 50 equality constraint problems from the CUTEst library, we report  results for three sets of experiments that exemplify the typical behavior observed in our more comprehensive set of experiments. 
The computations were performed on a high-performance workstation with the following specifications: 16 Intel(R) Xeon(R) Silver 4112 CPUs @ 2.60GHz, running on a Linux operating system, and equipped with 200 GB of RAM.


\subsection{Ability to Recover from Small Trust Region}

One potential weakness of trust region methods in a noisy environment occurs when the  radius becomes too small with respect to the noise level in the problem. The iteration may then reject steps, decreasing the trust radius further and ultimately terminating due to lack of progress. To demonstrate this behavior, we 
used problem HS7 and set the initial trust region radius $\Delta_0=10^{-7}$. The noise level was set to $\epsilon_g = \epsilon_A=\epsilon_f=\eps_c=0.1$, roughly a 0.033 relative error compared to the optimality and feasibility gaps at the starting point. 

We report the results in  \cref{HS7}, where the horizontal axis always indicates the number of iterations. 
We conducted three different experiments, superimposing the results to better contrast their differences. (i) We first report the performance of BO when noise is not injected into the functions (solid blue line). This was done by running the unmodified {\sc knitro} code. 
(ii) Next, we introduce noise into the problem but still used the unmodified {\sc knitro} code (i.e. the standard BO method). The results are depicted by the solid orange line.
(iii) Finally, we present the results of {\sc knitro} with our proposed modification as described in Algorithm~1 (solid green line).
We plot a horizontal red dashed line that marks the optimal objective value plus the noise level. 

\begin{figure}[h!]
	\begin{center}
		\includegraphics[width=.7\textwidth]{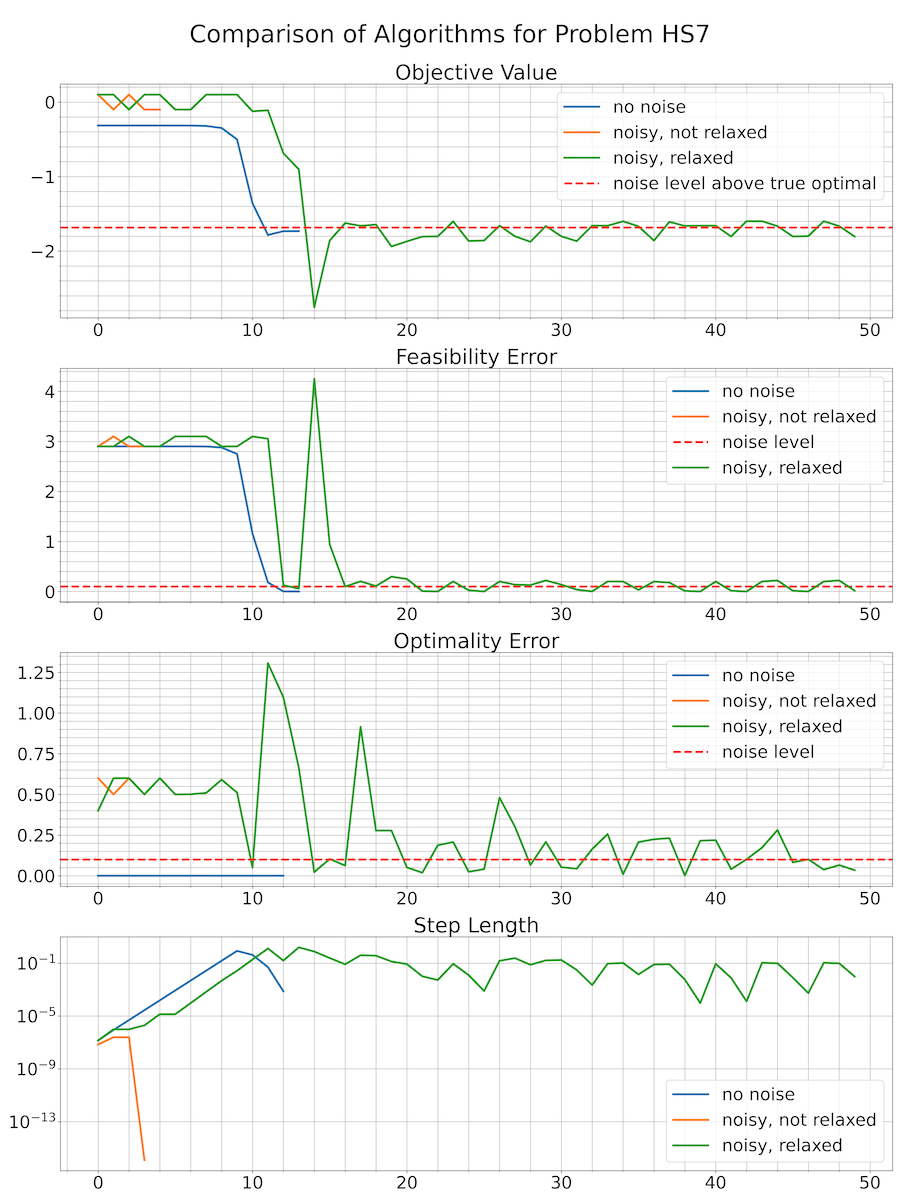}
		\caption{Testing the Byrd-Omojokun algorithm with and without noise, and the modified method.
  }
		\label{HS7}
	\end{center}
\end{figure}
Figure~\ref{HS7}
includes 4 plots reporting the objective function value, feasibility error {$\| c(x_k)\|$, optimality error $\| A_k \lambda_k - g(x_k)\|$, and step length $\| x_{k+1}-x_k\|$}. As can be observed, when the initial trust region radius is small, the unmodified algorithm (orange line) fails to converge because the trust region radius is driven to zero prematurely, while Algorithm~1 proceeds without difficulties.

\subsection{Premature Shrinkage of the Trust Region at Run Time}

{We have observed that the standard BO method may falter when $\Delta_0$ is very small. We now demonstrate that, starting with a sufficiently large trust radius, the algorithm can unnecessarily reduce the trust region during a run, leading to failure. We demonstrate with problem `ROBOT' from the CUTEst, with $\Delta_0=1$, and repeat the set of three experiments as before. The results are presented in \cref{fig:ROBOT}. }


\begin{figure}[h!]
	\begin{center}
		\includegraphics[width=.65\textwidth]{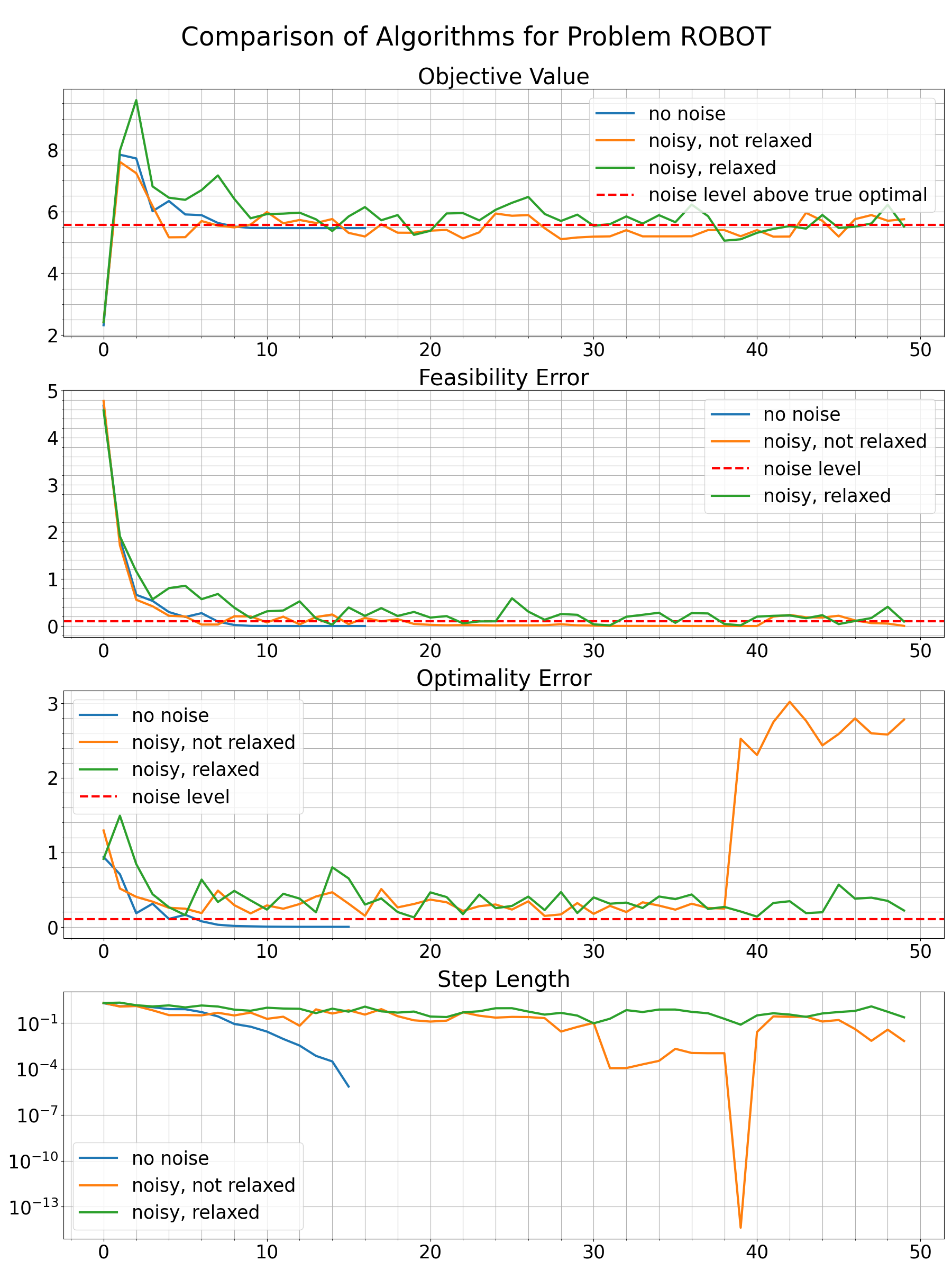}
		\caption{Performance of the algorithms with a sufficiently initial trust region}
		\label{fig:ROBOT}
	\end{center}
\end{figure}

As observed in figure~\ref{fig:ROBOT}, the proposed Algorithm~1 was able to reduce both feasibility and optimality below the noise level, whereas the unmodified algorithm starts shrinking the trust region radius (at around iteration 39) after rejecting many steps due to noise.
Even employing heuristics that restart the trust region, the algorithm makes wrong decisions that result in sharp increases in optimality error. Similar results have been observed for many other test problem during our experimentation. 


\subsection{The Cases where the Unmodified Algorithm  Performs Well}
There are test cases where the unmodified BO algorithm performs well, 
as illustrated in \cref{fig:BYRDSPHR}.
We observe that with noise, both the modified and unmodified algorithms were able to reduce the objective function, feasibility error, and optimality error below the noise level-- although the unmodified algorithm required more iterations and exhibited more oscillations. The unmodified algorithm terminated when the trust region became very small, a behavior that is in fact desirable when the iterates have already reached below the noise level. However, this behavior is brittle, because if shrinkage of the trust region occurs earlier, it can result in a failure to converge as seen above. We conclude from our experiments that the modified algorithm is preferred.

\begin{figure}[h!]
	\begin{center}
		\includegraphics[width=.65\textwidth]{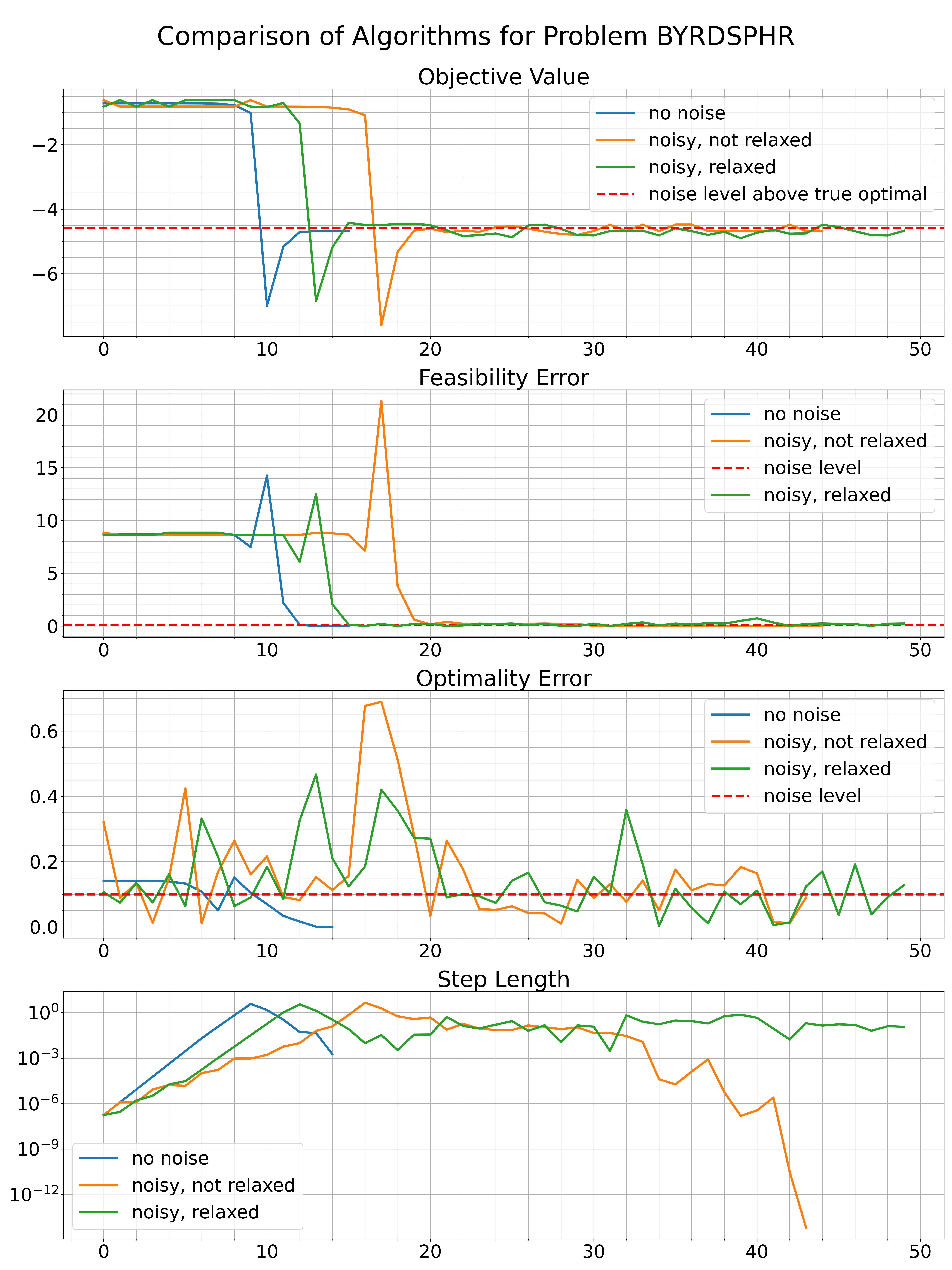}
		\caption{Cutest Problem BYRDSPHR, Initialized with TR = $10^{-7}$}
		\label{fig:BYRDSPHR}
	\end{center}
\end{figure}


\newpage
\section{Final Remarks}

When adapting the Byrd-Omojokun method to problems where the noise level can be estimated, it is not necessary to change the penalty parameter update rule or other components of the algorithm. Only the ration test  \eqref{relaxed} must be safeguarded. This paper presents a comprehensive convergence theory of the noise-tolerant BO method.  The analysis is complex due to the memory nature of trust region methods.
The proposed method has been implemented in the {\sc knitro} software package, and the numerical results reinforce our theoretical findings.



\bigskip\noindent\textit{Acknowledgements.} We are grateful to Richard Byrd for many useful insights and suggestions;  to Figen Oztoprak for implementing the new algorithm within {\sc knitro}; and to  Richard Waltz for help and support during this research. We thank Yuchen Lou for proof reading and useful suggestions.

\bibliographystyle{spmpsci}
\bibliography{references}

\end{document}